\documentclass[a4paper]{amsart}
\usepackage[latin1]{inputenc}
\usepackage{amssymb}
\usepackage{amsmath}
\usepackage{mathrsfs}
\usepackage{eufrak}
\usepackage{amsthm}
\usepackage{amsfonts}
\usepackage{textcomp}
\usepackage{graphicx}
\usepackage[pdftex]{color}
\usepackage{paralist}
\usepackage[shortlabels]{enumitem}
\usepackage[hyperpageref]{backref}
\usepackage[pagebackref]{hyperref}
\renewcommand*{\backref}[1]{}  
   \renewcommand*{\backrefalt}[4]{
      \ifcase #1 
         Not cited.
      \or
         Cited on page #2.
      \else
         Cited on pages #2.
      \fi}
\usepackage{comment}
\usepackage[arrow, matrix, curve]{xy}
\usepackage{tikz}
\usepackage{tikz-cd}
\usetikzlibrary{patterns,shapes,math,decorations.pathmorphing,decorations.pathreplacing,calc,arrows}

\newtheorem*{corollary*}{Corollary}
\newtheorem{theorem}{Theorem}[section]

\newtheorem{corollary}[theorem]{Corollary}
\newtheorem{lemma}[theorem]{Lemma}
\newtheorem{definition-proposition}[theorem]{Definition-Proposition}
\newtheorem{proposition}[theorem]{Proposition}
\newtheorem{question}[theorem]{Question}

\newtheorem*{claim*}{Claim}

\newtheorem{problem}[theorem]{Problem}

\theoremstyle{definition}
\newtheorem{definition}[theorem]{Definition}

\newtheorem*{theorem }{Theorem}
\newtheorem{remark}[theorem]{Remark}
\newtheorem{example}[theorem]{Example}

\theoremstyle{remark}

\numberwithin{equation}{theorem}

\renewcommand{\mod}{\operatorname{mod}}
\newcommand{\proj}{\operatorname{proj}}

\newcommand{\id}{\operatorname{id}}
\newcommand{\DN}{\operatorname{DN}}
\newcommand{\gr}{\operatorname{grade}}
\newcommand{\sgr}{\operatorname{s.grade}}

\newcommand{\cogr}{\operatorname{cograde}}
\newcommand{\Dom}{\operatorname{Dom}}

\newcommand{\LL}{\operatorname{LL}}

\newcommand{\Ext}{\operatorname{Ext}}
\newcommand{\gldim}{\operatorname{gldim}}

\newcommand{\Gr}{\operatorname{Gr}}
\newcommand{\SGr}{\operatorname{SGr}}
\newcommand{\End}{\operatorname{End}}
\newcommand{\Hom}{\operatorname{Hom}}
\newcommand{\add}{\operatorname{\mathrm{add}}}
\renewcommand{\top}{\operatorname{\mathrm{top}}}
\newcommand{\rad}{\operatorname{\mathrm{rad}}}
\newcommand{\Cokernel}{\operatorname{Cok}\nolimits}
\newcommand{\Kernel}{\operatorname{Ker}\nolimits}
\newcommand{\ind}{\operatorname{ind}\nolimits}

\newcommand{\soc}{\operatorname{\mathrm{soc}}}

\renewcommand{\mod}{\operatorname{mod}}

\newcommand{\Z}{\mathbb{Z}}
\newcommand{\R}{\mathbb{R}}
\newcommand{\D}{\mathcal{D}}
\newcommand{\RHom}{\operatorname{\mathsf{R}Hom}}

\newcommand{\Db}{\mathrm{D}^{\mathrm{b}}}

\newcommand{\CC}{\mathcal{C}}

\newcommand{\ZZ}{\mathcal{Z}}

\newcommand{\domdim}{\operatorname{domdim}}
\newcommand{\inj}{\operatorname{inj}}
\newcommand{\codomdim}{\operatorname{codomdim}}

\newcommand{\idim}{\operatorname{idim}}
\newcommand{\pdim}{\operatorname{pdim}}
\newcommand{\CM}{\operatorname{CM}}
\newcommand{\Tor}{\operatorname{Tor}}

\newcommand{\op}{\mathrm{op}}
\begin{document}

\title[Dominant Auslander-Gorenstein algebras and mixed cluster tilting]{Dominant Auslander-Gorenstein algebras and\\ mixed cluster tilting}\date{\today}

\dedicatory{Dedicated to the memory of Hiroyuki Tachikawa}

\author[Chan]{Aaron Chan}
\address[Chan]{Graduate School of Mathematics, Nagoya University, Furocho, Chikusaku, Nagoya 464-8602, Japan}
\email{aaron.kychan@gmail.com}

\author[Iyama]{Osamu Iyama}
\address[Iyama]{Graduate School of Mathematical Sciences, University of Tokyo, 3-8-1 Komaba Meguro-ku Tokyo 153-8914, Japan}
\email{iyama@ms.u-tokyo.ac.jp}

\author[Marczinzik]{Ren\'{e} Marczinzik}
\address[Marczinzik]{Mathematical Institute of the University of Bonn, Endenicher Allee 60, 53115 Bonn, Germany}
\email{marczire@math.uni-bonn.de}

\thanks{The first author is supported by JSPS Grant-in-Aid for Research Activity Start-up program 19K23401.  The second author is supported by JSPS Grant-in-Aid for Scientific Research (B) 22H01113 and (C) 18K3209.  The last named author was supported by the DFG with the project number 428999796.}

\begin{abstract}
We introduce the class of dominant Auslander-Gorenstein algebras as a generalisation of higher Auslander algebras and minimal Auslander-Gorenstein algebras, and give their basic properties.
We also introduce mixed (pre)cluster tilting modules as a generalisation of (pre)cluster tilting modules, and establish an Auslander type correspondence by showing that dominant Auslander-Gorenstein (respectively, Auslander-regular) algebras correspond bijectively with mixed precluster (respectively, cluster) tilting modules.
We show that every trivial extension algebra $T(A)$ of a $d$-representation-finite algebra $A$ admits a mixed cluster tilting module and show that this can be seen as a generalisation of the well known result that $d$-representation-finite algebras are fractionally Calabi-Yau. 
We show that iterated SGC-extensions of a gendo-symmetric dominant Auslander-Gorenstein algebra admit mixed precluster tilting modules.
\end{abstract}

\maketitle
\tableofcontents
\section*{Introduction}
The notion of Gorenstein rings is central in algebra. They are commutative Noetherian rings $R$ whose localisation at prime ideals have finite injective dimension. One of the characterisations of Gorenstein rings due to Bass is that, for a minimal injective coresolution
\[0\to R\to I^0\to I^1\to\cdots\] 
of the $R$-module $R$ and each $i\ge0$, the term $I^i$ is a direct sum of the injective hull of $R/p$ for all prime ideals $p$ of height $i$. In this case, $I^i$ has flat dimension $i$, and this property plays an important role in the study of non-commutative analogues of Gorenstein rings. A noetherian ring $A$, which is not necessarily commutative, satisfies the Auslander condition if there exists an injective coresolution
\begin{equation}\label{I^i}
0\to A\to I^0\to I^1\to\cdots
\end{equation}
of the right module $A$ such that the flat dimension of $I^i$ is bounded by $i$ for all $i\ge0$. This condition is left-right symmetric, and there are many other equivalent conditions. If, additionally, $A$ has finite injective dimension, $A$ is said to be \emph{Auslander-Gorenstein}, and if $A$ has finite global dimension, $A$ is said to be \emph{Auslander-regular}.
These algebras play important roles in various areas, including homological algebra, non-commutative algebraic geometry, analytic $D$-modules, Lie theory and combinatorics, see for example \cite{AR,B,C,IM,KMM,VO}. 

\medskip
The classical Auslander correspondence \cite{A,ARS} is a milestone in representation theory which gives a bijective correspondence between representation-finite artin algebras $B$ and \emph{Auslander algebras} $A$, which are by definition the artin algebras $A$ with $\gldim A \leq 2 \leq \domdim A$, where $\domdim X$ of an $A$-module $X$ with minimal injective coresolution
\[0\to X\to I^0\to I^1\to\cdots\]
is the infimum of $i\ge0$ such that $I^i$ is non-projective. This was one of the starting points of Auslander-Reiten theory developed in the 1970s. It encodes the representation theory of $B$ into the homological algebra of its Auslander algebra $A$.
More generally, algebras $A$ satisfying $\gldim A\leq d\leq \domdim A$ are called \emph{$(d-1)$-Auslander algebras}
for each $d\ge0$, and there exists a bijective correspondence (called the higher Auslander correspondence) between \emph{$d$-cluster tilting modules} $M$ over finite dimensional algebras $B$ (see \eqref{define CT}) and $d$-Auslander algebras for each $d\geq 1$ given by $M\mapsto A=\End_B(M)$ \cite{I6}.
Finite dimensional algebras $B$ with $\gldim A \leq d$ admitting $d$-cluster tilting modules are called \emph{$d$-representation finite}. These notions are central in higher Auslander-Reiten theory of finite dimensional algebras and Cohen-Macaulay representations \cite{I5}. 

A further generalisation of the Auslander correspondence to algebras with infinite global dimension was established recently in \cite{IyaSol}, building upon \cite{AS}. 
In this setting, Auslander algebras are replaced by \emph{minimal Auslander-Gorenstein algebras}, which are algebras $A$ with $\idim A \leq d \leq \domdim A$ for some $d \geq 0$.  
Similar to the higher Auslander correspondence, we have $A\simeq \End_B(M)$ for some finite dimensional algebra $B$ and a \emph{precluster tilting $B$-module} $M$.
 Several applications and interactions with other fields are found; these include cluster algebras of Fomin-Zelevinsky \cite{FZ} and non-commutative crepant desingularisations in algebraic geometry \cite{V}, see also \cite{AT,DI,DJL,DJW,GI,IO1,IO2,JKM,H,HS,HI1,HI2,HZ1,HZ2,HZ3,JK,M,P,S,ST,Va,W,Wu} and
\cite{CIM,CheKoe,DITW,HKV,Grev,LMZ,LZ,MS,MMZ,Mc,NRTZ,PS,R,Z}. 

The first aim of this paper is to introduce \emph{dominant Auslander-Gorenstein algebras}, a generalisation of higher Auslander algebras as well as minimal Auslander-Gorenstein algebras.  A dominant Auslander-Gorenstein algebra $A$ is an Iwanaga-Gorenstein algebra $A$ such that $\idim P \leq \domdim P$ holds for every indecomposable projective $A$-module $P$.
A \emph{dominant Auslander-regular algebra} is then defined to be a dominant Auslander-Gorenstein algebra with finite global dimension.  We refer to Theorem \ref{conditionGtheorem} and Section \ref{sec:MPCT} for other equivalent characterisations of dominant Auslander-regular algebras.
\[\xymatrix@R1.5em{
\mbox{higher Auslander}\ar@{=>}[r]\ar@{=>}[d]&\mbox{dominant Auslander-regular}\ar@{=>}[r]\ar@{=>}[d]&\mbox{Auslander-regular}\ar@{=>}[d]\\
\mbox{minimal Auslander-Gorenstein}\ar@{=>}[r]&\mbox{dominant Auslander-Gorenstein}\ar@{=>}[r]&\mbox{Auslander-Gorenstein}}
\]
The class of dominant Auslander-Gorenstein algebras is much larger than that of minimal Auslander-Gorenstein algebras, and still enjoys extremely nice homological properties among all Auslander-Gorenstein algebras.
One of the advantages of dominant Auslander-regular algebras is that they are closed under gluing of algebras (Section \ref{section: glueing}) and Koszul duality \cite{CIM2} while higher Auslander algebras are not so in general.

One of our main results gives the Auslander correspondence for dominant Auslander-Gorenstein algebras. For this, we generalise the notion of cluster tilting modules as follows: For an algebra $B$, we consider the higher Auslander-Reiten translates
\[\tau_n:=\tau \Omega^{n-1}\ \mbox{ and }\ \tau_n^{-}:=\tau^{-1} \Omega^{-(n-1)}.\]
Then a generator-cogenerator $M$ of $B$ is called \emph{mixed precluster tilting} if the following condition is satisfied.
\begin{enumerate}
\item[$\bullet$] For each indecomposable non-projective direct summand $X$ of $M$, there exists $\ell_X\ge1$ such that $\Ext^i_B(X,M)=0$ for all $1\le i<\ell_X$ and $\tau_{\ell_X}(X)\in\add M$.
\end{enumerate}
This is equivalent to its dual condition, see Definition-Proposition \ref{mixed 0}.
We refer to Sections \ref{section: TrivExtdrep}, \ref{section:example mixed pct} and \ref{section: glueing} for examples of mixed precluster tilting modules.
We prove the following Auslander correspondence for dominant Auslander-Gorenstein algebras.

\begin{theorem}[Theorem \ref{main correspondence 0}]\label{main correspondence 0 in intro}
There exists a bijection between the following objects.
\begin{enumerate}[\rm(1)]
\item The Morita equivalence classes of dominant Auslander-Gorenstein algebras $A$ with $\domdim A\ge 2$.
\item The Morita equivalence classes of pairs $(B,M)$ of finite dimensional algebras $B$ and mixed precluster tilting modules $M$.
\end{enumerate}
The correspondence from {\rm(2)} to {\rm(1)} is given by $(B,M)\mapsto A:=\End_B(M)$.
\end{theorem}

To a mixed precluster tilting module $M$, we associate the subcategory
\[\mathcal{Z}(M):=\bigcap\limits_{X\in\ind(\add M)}X^{\perp{\ell_X-1}}\]
of $\mod B$ (Definition \ref{ZMdefinition}), and we prove that there exists an equivalence
\[\Hom_B(M,-):\mathcal{Z}(M)\simeq\CM A,\] 
where $\CM A$ denotes the subcategory of maximal Cohen-Macaulay $A$-modules (Theorem \ref{thm-CM}). It follows that $A$ is dominant Auslander-regular if and only if $\mathcal{Z}(M)=\add M$.
Motivated by this result, we call a mixed precluster tilting module $M$ \emph{mixed cluster tilting} if it satisfies $\mathcal{Z}(M)=\add M$.
Thus we have implications
\[\xymatrix@R1.5em{
\mbox{$d$-cluster tilting}\ar@{=>}[r]\ar@{=>}[d]&\mbox{mixed cluster tilting}\ar@{=>}[d]\\
\mbox{$d$-precluster tilting}\ar@{=>}[r]&\mbox{mixed precluster tilting.}}
\]
We refer to Sections \ref{section: TrivExtdrep} and \ref{section: glueing} for examples of mixed cluster tilting modules.
Now we are ready to state the following restriction of the correspondence in Theorem \ref{main correspondence 0 in intro}.

\begin{theorem}[Theorem \ref{main correspondence 0 2}]\label{main correspondence 0 2 in intro}
The bijection in Theorem \ref{main correspondence 0 in intro} restricts to a bijection between the following objects.
\begin{enumerate}[\rm(1)]
\item The Morita equivalence classes of dominant Auslander-regular algebras $A$ with $\domdim A\ge 2$.
\item The Morita equivalence classes of pairs $(B,M)$ of finite dimensional algebras $B$ and mixed cluster tilting modules $M$.
\end{enumerate}
\end{theorem}

Recall that the \emph{complexity} of a module $M\in\mod A$ is defined as
\begin{align}
\inf\{ i\in\Z_{\ge0}\mid\exists c\in\R_{>0}\,\forall n\in\Z_{\ge0}\, : \dim_KP_n\le cn^{i-1}\},\label{eq:cx(M)}
\end{align}
where $\cdots\to P_1\to P_0\to M\to0$ is a minimal projective resolution of $M$.
The complexity of an algebra $A$ is the supremum of the complexity of its modules in $\mod A$. We prove the following generalisation of Erdmann and Holm's  result \cite{EH} on the complexity of self-injective algebras admitting $d$-cluster tilting modules.


\begin{theorem}[Theorem \ref{ErdmannHolmgeneralisation}]\label{complexity}
If a finite dimensional self-injective algebra $A$ admits a mixed cluster tilting module, then $A$ has complexity at most one.
\end{theorem}

The final two sections contain applications and examples of mixed (pre)cluster tilting modules.
Recall that the trivial extension algebra $T(A)=A \oplus DA$ of a finite dimensional algebra $A$ is a Frobenius algebra that is of central importance in representation theory.
Trivial extension algebras as well as its covering called repetitive algebras give a systematic construction of self-injective algebras using orbit categories, and it is known that any representation-finite self-injective algebras are obtained in this way \cite{BLR,Bon,Rie}. It also plays an important role in 
a recent characterisation of twisted fractionally Calabi-Yau algebras of finite global dimension using trivial extension algebras in \cite{CDIM}.
We show that, for each Dynkin quiver $Q$, the trivial extension algebra $B=T(KQ)$ has a mixed cluster tilting module given by a direct sum of $B$ and all indecomposable non-projective $KQ$-modules. In fact, our result is much more general and can be stated as follows.

\begin{theorem}[Theorem \ref{B+N}]\label{B+N in intro}
Let $A$ be a $d$-representation-finite algebra with a basic $d$-cluster tilting module $M=A\oplus N$ and let $B=T(A)$ be the trivial extension of $A$.
Then $B \oplus N$ is a mixed cluster tilting $B$-module.
\end{theorem}

This result is interesting from the point of view of  Darp\"o and Kringeland's  classification \cite{DK} of cluster tilting modules over the trivial extension of Dynkin path algebras.  It turns out that, for $d\geq 2$, $d$-cluster tilting modules over this class of algebras only exist in Dynkin types $A_n$ and $D_4$; on the other hand, Theorem \ref{B+N in intro} tells us that non-trivial mixed cluster tilting modules exist for any Dynkin types.

In Section \ref{section:example mixed pct}, we discuss an inductive way to obtain infinitely many mixed cluster tilting modules starting from algebras with certain properties, we refer to Theorem \ref{thm: iterative construction} for full details. We just state here the main result for the most important special case, namely gendo-symmetric algebras.
Recall that a finite dimensional algebra $A$ is called \emph{gendo-symmetric} if $A \simeq \End_B(M)$ for a symmetric algebra $B$ with generator $M$ of $\mod B$. Gendo-symmetric algebras were introduced by Fang and Koenig in \cite{FanKoe} and containing important classes of algebras such as Schur algebras $S(n,r)$ for $n \geq r$, blocks of category $\mathcal{O}$ and centraliser algebras of nilpotent matrices, see for example \cite{KSX} and 
\cite{CrM}.
On the other hand, the \emph{SGC-extension} (smallest generator-cogenerator extension) of an algebra $A$ is $\End_A(A\oplus DA)$, see e.g. \cite{CIM}.
The main result of Section \ref{section:example mixed pct} can be stated as follows for gendo-symmetric algebras and their iterated SGC-extensions.

\begin{theorem}[Corollary \ref{gendosymtheorem}]
Let $A=A_0$ be a gendo-symmetric algebra. For each $i\ge0$, let $M_i:=A_i \oplus DA_i$ and $A_{i+1}:=\End_{A_i}(M_i)$.
\begin{enumerate}[\rm(1)]
\item Assume that $A$ is a dominant Auslander-Gorenstein algebra. Then for each $i\ge0$, $A_i$ is a dominant Auslander-Gorenstein algebra with a mixed precluster tilting module $M_i$.
\item Assume that $A$ is a minimal Auslander-Gorenstein algebra. Then for each $i\ge0$, $A_i$ is a minimal Auslander-Gorenstein algebra with a precluster tilting module $M_i$.
\end{enumerate}
\end{theorem}


In forthcoming work \cite{CIM2} the notion of mixed cluster tilting modules and dominant Auslander-regular algebras will be of central importance to answer a question of Green in \cite{G} about the homological characterisation of the Koszul dual algebras of Auslander algebras. We will see that Koszul duality leads in a very natural way to the notion of dominant Auslander-regular algebras when working with the classical higher Auslander algebras and this will also allow us to discover large new classes of $d$-representation-finite algebras and cluster tilting modules.

\section{Preliminaries}
We assume that all algebras are finite dimensional over a field $K$ and modules are finitely generated right modules unless stated otherwise.
We assume that the reader is familiar with the basics of representation theory of algebras and refer for example to the textbooks \cite{ARS,ASS,SkoYam}.
For background on homological dimensions such as the dominant dimension and related properties we refer for example to \cite{Ta,Yam}. For background on Gorenstein homological algebra and maximal Cohen-Macaulay modules we refer to \cite{AB,Che}.

Throughout, $D=\Hom_K(-,K)$ denotes the natural duality of the module category of an algebra $A$, $\proj A$ denotes the full subcategory of projective $A$-modules and $\inj A$ the full subcategory of injective $A$-modules.
We will usually omit the bracket and write $DM$ when applying the duality $D$ to a module $M$, unless there is the danger of confusion -- for instance, we will write $D(Af)$ for injective $A$-module corresponding to an idempotent $f$.
A module $M$ is called \emph{basic} if it has no direct summand of the form $L \oplus L$ for some indecomposable module $L$. An algebra $A$ is called \emph{basic} if $A$ is a basic module. We denote by $\LL(M)$ the Loewy length of a module $M$.

\medskip

Denote by $\pdim M$ and $\idim M$ the projective dimension and injective dimension of a module $M$ respectively.
An algebra $A$ is said to be \emph{Iwanaga-Gorenstein} if $\idim A+\pdim DA < \infty$; in such a case, we have $\idim A=\pdim DA$ and $\idim A$ is called the \emph{self-injective dimension} of $A$.  A module $M$ with $\Ext_A^{i}(M,A)=0$ for all $i>0$ over an Iwanaga-Gorenstein algebra $A$ is called \emph{maximal Cohen-Macaulay} (also known as Gorenstein projective in the literature). 
Denote by $\CM A$ the full subcategory of maximal Cohen-Macaulay modules of $A$ and by $\Omega^n(\mod A)$ the full subcategory of modules that are isomorphic to the $n$-th syzygy of a module. 
One can show that an algebra $A$ is Iwanaga-Gorenstein with self-injective dimension $n$ if and only if $\Omega^n(\mod A)=\CM A$, see for example \cite[Theorem 2.3.3]{Che}.  Note that for an Iwanaga-Gorenstein algebra $A$, one has $\gldim A<\infty$ if and only if $\CM A=\proj A$; in which case, the global dimension coincides with the self-injective dimension.

The \emph{dominant dimension} $\domdim M$ of a module $M$ with a minimal injective coresolution 
$$0 \rightarrow M \rightarrow I^0 \rightarrow I^1 \rightarrow \cdots$$
is defined as the smallest $n$ such that $I^n$ is non-projective.
The \emph{codominant dimension} $\codomdim M$ of a module $M$ is defined as the dominant dimension of the left module $D(M)$.  $\Dom_n(A)$ denotes the full subcategory of $\mod A$ consisting of modules having dominant dimension at least $n$.
The dominant dimension of an algebra is defined as the dominant dimension of $A$ as an $A$-module, i.e. $\domdim A:=\domdim(A_A)$.  We note that $\domdim A=\codomdim DA=\domdim A^\op$ always holds.
We can characterise modules of large dominant dimension as follows.

\begin{lemma}\cite[Proposition 4]{MarVil} \label{marvillemma}
Let $A$ be an algebra of dominant dimension equal to $d \geq 1$. Then $\Dom_i(A)=\Omega^i(\mod A)$ for all $0\leq i\leq d$.
\end{lemma}

Let $A$ be a finite dimensional algebra, and let $e,f\in A$ be idempotents such that $D(Af)$ and $eA$ are the basic additive generators of projective-injective modules, i.e. 
\begin{align}\label{eAf}
\add D(Af)=\proj A\cap \inj A =\add(eA).
\end{align}
Then we have an isomorphism $fAf \simeq eAe$ of algebras.  
Having $\domdim A\ge 1$ is equivalent to saying that there is an injective map $0\to A\to I$ with $I\in \add D(Af)$, which is also equivalent to (by left-right symmetry) the existence of surjective map $P\to DA\to 0$ with $P\in \add (eA)$.
In classical ring theoretic language, this amounts to say that  $eA$ is a minimal faithful projective-injective $A$-module and $Af$ is a minimal faithful projective-injective  $A^\op$-module.  
We call $fAf\simeq eAe$ the \emph{base algebra} of $A$

The Morita-Tachikawa correspondence \cite{A,Yam} gives a bijection between Morita equivalence classes of algebras $A$ with $\domdim A\ge2$ and Morita equivalence classes of pairs $(B,M)$ of algebras $B$ and generator-cogenerators $M$ of $B$.  The map $(B,M)\mapsto A$ is given by the following proposition.

\begin{proposition}\cite{A}\label{From B to A}
Let $B$ be an algebra, $M$ a generator-cogenerator of $B$, and $A:=\End_B(M)$.
\begin{enumerate}[\rm(1)]
\item $\domdim A\ge2$.
\item Let $e,f$ be idempotents of $A$ such that $eA=\Hom_B(M,DB)$ and $fA=\Hom_B(M,B)$. Then \eqref{eAf} is satisfied.
\end{enumerate}
\end{proposition}
The map $A\mapsto(B,M)$ is given by the following proposition.

\begin{proposition}\cite{A}\label{prop:DCP}
For an algebra $A$ with $\domdim A\ge 2$, let $e,f\in A$ be idempotents satisfying \eqref{eAf}, and $B:=fAf$. 
Then the following hold.
\begin{enumerate}[\rm (1)]
\item $M:=Af$ is a generator-cogenerator over $B$, i.e. $B\oplus DB\in \add M$.
\item We have an isomorphism $A\to\End_B(M)$ of algebras sending $a\in A$ to $(a\cdot)\in\End_B(M)$.
\item For $X\in\ind(\add M)$, the indecomposable projective $A$-module $\Hom_B(M,X)$ is injective if, and only if, $X\in\inj B$.  In such a case, we have $\Hom_B(M,X)\simeq D\Hom_B(\nu^-(X),M)$.
\end{enumerate}
\end{proposition}

Under the setting of Proposition \ref{prop:DCP}, the following $(B,A)$-bimodules are isomorphic, and are additive generators of $\proj A\cap\inj A$:
\[
D(Af) \simeq eA\simeq  \Hom_{B}(M, DB) \simeq \Hom_{fAf} (Af, D(Af)f)\simeq D\Hom_{B}(B, M) \simeq  D\Hom_{fAf}(eAf,Af).
\]

\begin{proof}
For convenience of the reader, we include a proof.

(1) The $B$-module $B=fAf$ belongs to $\add Af=\add M$.
Since $\add D(Af)=\add eA$ as subcategories of $\mod A$, we have $\add D(Af)f=\add eAf$ as subcategories of $\mod B$.
Thus the $B$-module $DB=D(fAf)=(D(Af))f$ belongs to $\add eAf\subset\add M$.

(2) Take an injective resolution $0\to A\to J^0\to J^1$ with $J^i\in\add D(Af)$. Multiplying $f$ from the right, we have an exact sequence $0\to M\to J^0f\to J^1f$ of $B$-modules.
Applying $\Hom_B(M,-)=\Hom_B(Af,-)$ to the second sequence and comparing with the first sequence, we have a commutative diagram of exact sequences
\[\xymatrix@R1em{
0\ar[r]&\End_B(M)\ar[r]&\Hom_B(Af,J^0f)\ar[r]\ar@{=}[d]&\Hom_B(Af,J^1f)\ar@{=}[d]\\
0\ar[r]&A\ar[r]&J^0\ar[r]&J^1,
}\]
where we used $\Hom_B(Af,D(Af)f)=D(Af)$. Thus $\End_B(M)\simeq A$.

(3) is a direct consequence of the properties of the inverse Nakayama functor, see for example \cite[Chapter III. Proposition 2.11]{ASS}.
\end{proof}

The following is a special case of results from \cite{APT}.
\begin{proposition} \label{APTpropo}
Let $A$ be an algebra with $\domdim A\ge 1$ and $B=fAf$ be its base algebra.
\begin{enumerate}[\rm(1)]
\item The adjoint pair $(-\otimes_A A f,\Hom_B(A f,-))$ yields equivalences $\Dom_2(A) \simeq \mod B$, $\proj A\simeq \add Af$ and $\proj A\cap\inj A \simeq \inj B$.
\item If $M\in \Dom_{k+1}(A)$, then $\Ext_A^i(X,M)\simeq \Ext_B^i(Xf,Mf)$ for all $i\in[1, k-1]$ and any $X\in \mod A$.
\item $\domdim_A( \Hom_B(Af,X) )=\inf \{ i \geq 1 \mid  \Ext_B^i(Af,X) \neq 0 \} +1$.
\end{enumerate}
\end{proposition}
\begin{proof}
(1) is \cite{APT}, lemma 3.1, (2) is \cite{APT}, theorem 3.2, and (3) is \cite{APT}, proposition 3.7.
\end{proof}


In this article, a \emph{higher Auslander algebra} (more explicitly, \emph{$(d-1)$-Auslander algebra}) is an algebra $A$ with $\gldim A \leq d \leq \domdim A$ for some $d \geq 0$.
Note also that a basic $0$-Auslander algebra is precisely
a direct product of upper triangular matrix algebras over division rings, see \cite{I7}.
Generalising higher Auslander algebras, a \emph{minimal Auslander-Gorenstein algebra} (more explicitly, \emph{$(d-1)$-minimal Auslander-Gorenstein algebra}) is an algebras $A$ with $\idim A \leq d \leq \domdim A$ for some $d \geq 1$ \cite{IyaSol}. 
We refer also to \cite{CheKoe} for an alternative approach to minimal Auslander-Gorenstein algebras and \cite{CIM} for the construction and many examples of such algebras. 

For a subcategory $\CC$ of $\mod B$ we define for an integer $n$ the following orthogonal categories
\begin{align}\label{define perp}\notag
^{\perp n}\CC&:= \{ X \in \mod B \mid  \mbox{$\Ext_B^i(X,Y)=0$ for all $Y \in\CC$ and $1\le i\le n$}\},\\
\CC^{\perp n}&:= \{ X \in \mod B \mid  \mbox{$\Ext_B^i(Y,X)=0$ for all $Y \in\CC$ and $1\le i\le n$}\}.
\end{align}
For example, $^{\perp -1}\CC={}^{\perp 0}\CC=\CC^{\perp 0}=\CC^{\perp -1}=\mod B$. Note that when $\CC=\add M$ for some module $M\in\mod B$ we often write $M^{\perp n}$ or $^{\perp n}M$ instead of $\add M^{\perp n}$ or $^{\perp n}\add M$. 
For $d\ge1$, we call $M\in\mod B$ \emph{$d$-precluster tilting} if
\begin{equation}\label{define preCT}
M^{\perp d-1}={}^{\perp d-1}M.
\end{equation}
This is equivalent to saying that $\End_B(M)$ is a $d$-minimal Auslander-Gorenstein algebra, see \cite{IyaSol}.
If, in addition, the above orthogonal categories coincide with $\add M$, that is,
\begin{equation}\label{define CT}
\add M=M^{\perp d-1}={}^{\perp d-1}M
\end{equation}
holds, then we call $M$ a \emph{$d$-cluster tilting} module.
This is equivalent to saying that $\End_B(M)$ is a $d$-Auslander algebra.
By (pre)cluster tilting module, we mean a $d$-(pre)cluster tilting module for some $d\ge 1$.

Denote by $\nu:=-\otimes_A DA$ the \emph{Nakayama functor} on $\mod A$, and by $\nu^-:=\Hom_A(DA,-)$ its adjoint.  Note that these yield $\nu:\proj A\simeq \inj A$.
Denote by $\tau$ the Auslander-Reiten translate and by $\tau^-$ the inverse Auslander-Reiten translate.  Note that 
\begin{align}\label{eq:tau=derived nu}
\tau(M)=\Kernel(\nu(P'\xrightarrow{f} P)) \text{ and } \tau^-(M)=\Cokernel(\nu^-(I\xrightarrow{g} I'))
\end{align}
for $P'\xrightarrow{f} P\to M\to 0$ the projective presentation of $M$ and $0\to M\to I\xrightarrow{g} I'$ the injective copresentation of $M$.
For $d\ge 1$, denote by $\tau_d:=\tau \Omega^{d-1}$ and $\tau_d^{-}:=\tau^{-1} \Omega^{-(d-1)}$.

Suppose $\CC$ is a full subcategory of $\mod A$ containing $A$ (resp. $DA$).
Denote by $\underline{\CC}$ (resp. $\overline{\CC}$) the additive quotient of $\CC$ by the ideal consisting of all morphisms factoring through modules in $\add A$ (resp. $\add DA$).  In \cite[Theorem 1.4.1]{I5}, it is shown that for each $d\ge 1$, there is an equivalence of additive categories
\begin{equation}\label{eq:stable equiv perpA DAperp}
\begin{tikzcd}
\tau_d : \underline{{}^{\perp d-1} A} \ar[r,"\sim"]& \overline{DA^{\perp d-1}}: \tau_d^-
\end{tikzcd}
\end{equation}
This induces the higher Auslander-Reiten duality \cite[Theorem 1.5]{I5}
\begin{equation}\label{eq:AR dual}
\Ext_A^{d-i}(X,Z) \simeq D\Ext_A^{i}(Z,\tau_d X) \text{ and } \underline{\Hom}_A(X,Z):=\Hom_{\underline{\mod A}}(X,Z)\simeq D\Ext_A^{d}(Z,\tau_dX).
\end{equation}

\medskip

The \emph{grade} of a module $M$ is defined as $\gr M:= \inf \{i \geq 0 \mid  \Ext_A^i(M,A) \neq 0 \}$ and dually the \emph{cograde} is defined as $\cogr M:= \inf \{ i \geq 0 \mid  \Ext_A^i(DA,M) \neq 0 \}$. The \emph{strong grade} $\sgr M$ of a module $M$ is the infimum of $\gr N$ for all submodules $N$ of $M$.

\begin{definition}\label{define AG}\cite{FGR}
Let $A$ be an algebra and
\begin{equation}\label{injective coresolution of A}
0\to A\to I^0\to I^1\to\cdots
\end{equation}
a minimal injective coresolution of the $A$-module $A$.
We call $A$ \emph{Auslander-Gorenstein} when $A$ is Iwanaga-Gorenstein and additionally satisfies the following equivalent conditions.
\begin{enumerate}
\item $\pdim I^i \leq i$ for all $i \geq 0$.
\item For each $i\ge1$ and $X\in\mod A$, we have $\sgr\Ext^i_A(X,A)\ge i$.
\end{enumerate}
\end{definition}

When an Auslander-Gorenstein algebra $A$ has finite global dimension, $A$ is called \emph{Auslander-regular}. We refer for example to the survey article \cite{C} for more on Auslander-Gorenstein algebras.
An algebra $A$ is Auslander-Gorenstein if and only if its opposite algebra is Auslander-Gorenstein.  
Given an algebra $A$ with minimal injective coresolution \eqref{injective coresolution of A}, we define the set of \emph{dominant numbers} of $A$ to be
\[
\DN(A):=\{ n\geq 0 \mid \pdim I^i < \pdim I^n\text{ for all }0\leq i<n\}.
\]
Note that $0\in \DN(A)$ always.
For $l \geq 0$ we define the subcategories
\begin{align*}
\Gr_l(A)&:= \{ M \in \mod A \mid  \gr M \geq l \},\\
\SGr_l(A)&:= \{ M \in \mod A \mid  \sgr M \geq l \}.
\end{align*}
It is elementary that $\SGr_l(A)$ is a Serre subcategory of $\mod A$, that is, closed under subfactor modules and extensions.
We will need the following result.

\begin{proposition} \label{gradetheorem}
Suppose $A$ is an Auslander-Gorenstein algebra.  Then the following hold.
\begin{enumerate}[\rm(1)]
\item $\DN(A)=\DN(A^\op)=\{\gr X\mid X\in\mod A\}$.
\item $\Gr_l(A)=\SGr_l(A)$ and this subcategory is a Serre subcategory of $\mod A$ for all $0\leq l\leq \idim A$.
\end{enumerate}
\end{proposition}
\begin{proof}
(1) is \cite[Theorem 1.1]{I2}, (2) is a special case of \cite[Proposition 2.4]{I2} due to the discussion after \cite[Definition 2.2]{I2}.
\end{proof}

\section{Dominant Auslander-Gorenstein algebras and basic properties}

\subsection{The definition and the first properties}

\begin{definition}
An algebra $A$ is a \emph{dominant Auslander-Gorenstein algebra} if it is Iwanaga-Gorenstein and $\idim P \leq \domdim P$ for every indecomposable projective module $P$.
If, furthermore, $\gldim A<\infty$, then we call it a \emph{dominant Auslander-regular algebra}.
\end{definition}
Note that for a dominant Auslander-Gorenstein algebra $\idim P<\domdim P$ holds if and only if $P$ is also injective; otherwise, we have $\domdim P=\idim P$.
In particular, a dominant Auslander-Gorenstein algebra is self-injective if and only if we have strict inequality for all $P$.

\begin{proposition} \label{strongAGpropo}
Suppose $A$ is a dominant Auslander-Gorenstein algebra.  Then the following hold.
\begin{enumerate}[\rm(1)]
\item For every indecomposable projective $A$-module $P$ with $d:=\idim P\geq 0$, the minimal injective coresolution 
\begin{align*}
0\to P \to T^0 \to T^1 \to \cdots \to T^{d-1} \to I \to 0
\end{align*}
has indecomposable injective $I$ and projective-injective $T^k$ for all $0\leq k<d$.  In particular, this is also the the minimal projective resolution of $I$.
\item The map $P \mapsto I=\Omega^{-\idim P}(P)$ defines a bijection $\pi:\ind(\proj A)\to \ind(\inj A)$ with inverse $I\mapsto \Omega^{\pdim I}(I)$.
\item $A^\op$ is also dominant Auslander-Gorenstein.
\item $\DN(A)=\{\idim P\mid P\text{ indecomposable projective}\}$.
\item $A$ is Auslander-Gorenstein.
\end{enumerate}
\end{proposition}

\begin{proof}
By assumption $A$ is Iwanaga-Gorenstein. Let $n:=\idim A$ denote the self-injective dimension of $A$.  If $n=0$, then $A$ is self-injective and the claim is trivial; hence, we will assume $n>0$ from now on.
Label the indecomposable projective (resp. injective) $A$-modules by $P_1,\ldots, P_m$ (resp. $I_1,\ldots, I_m$) so that $\top P_i=\soc I_i$ for all $i$.

(1) When $P_i$ is injective, the claim is trivial.
Consider a non-injective $P_i$ with $d_i :=\domdim P_i=\idim P_i$, so there is the following minimal injective coresolution:
$$0 \rightarrow P_i \rightarrow T_i^0 \rightarrow T_i^1 \rightarrow \cdots \rightarrow T_i^{d_i} \rightarrow 0.$$
Since $P_i$ is non-injective, we have $\domdim P_i=\idim P_i$, which then means that the module $T_i^k$ are projective-injective for $0\leq k < d_i$. 
Now $T_i^{d_i}$ is injective and has no projective direct summands or else the resolution would split off this summand and would not be minimal. If $T_i^{d_i}=X_1 \oplus X_2$ for two non-zero direct summands $X_1$ and $X_2$ then those summands would have both codominant dimension at least $d_i$ by the above minimal coresolution of $P_i$. But then $\Omega^{d_i}(X_1)$ and $\Omega^{d_i}(X_2)$ would both be non-zero, contradicting that $P_i$ is indecomposable. Thus $T_i^{d_i}$ must be indecomposable.

(2) This follows from (1) immediately.

(3) By (1), we have $\codomdim I_{\pi(i)} \leq \pdim I_{\pi(i)}$ with strict inequality if and only if $I_{\pi(i)}\simeq P_i\in \proj A$.  By $K$-linear duality this means that $A^\op$ is dominant Auslander-Gorenstein.

(4) As $\pdim T_i^k=0$ for all $0\leq k< d_i$ and $\pdim T_i^{d_i} = \pdim I_{\pi(i)}=d_i$, the claim on $\DN(A)$ now follows.

(5) Since the minimal injective coresolution $(I^k)_{k\geq 0}$ of $A$ is the direct sum of those for $P_i$ over all $i$'s, we have $\pdim I^k =\pdim \bigoplus_i T_i^k$, which is $k$ if $k=d_i$, otherwise, $0$.  Thus, it immediately follows that $A$ is Auslander-Gorenstein.
\end{proof}

By Theorem \ref{gradetheorem}, let $A$ be a dominant Auslander-Gorenstein algebra with $\DN(A)=\{d_1<d_2<\cdots d_r\}$, then for each $1\leq i<r$, we have $\Gr_{d_i+1}(A)=\Gr_{d_i+2}(A)=\cdots =\Gr_{d_{i+1}}(A)$. 
In other words, knowledge of the subcategories $\Gr_{\idim P}(A)$ over all indecomposable projective $P$ gives complete information on the grades of modules.

\begin{proposition} \label{gradepropo}
Let $A$ be a dominant Auslander-Gorenstein algebra, $0 \rightarrow A \rightarrow I^0 \rightarrow I^1 \rightarrow \cdots $ a minimal injective coresolution of the regular module $A$, and  $\DN(A)= \{0=d_0<d_1<\cdots<d_r\}$. 
Define the idempotents $f_i$ so that $D(Af_i)$ is the basic additive generator of $\add(\bigoplus_{0\leq j\leq i}{I^{d_j}})$ for all $0\leq i\leq r$.
Then $\Gr_{d_i}(A) \simeq \mod A/Af_{i-1}A$ for all $i$ with $f_{-1}:=0$ 
\end{proposition}

\begin{proof}
By Theorem \ref{gradetheorem} (b), for any $0\leq t\leq \idim A$, the subcategory $\Gr_t(A)$ is a Serre subcategory and so
there is an idempotent $e_t\in A$ so that $\Gr_t(A) \simeq \mod A/Ae_tA$.

For $d=0\in \DN(A)$, it is clear that $\Gr_0(A)=\mod A$.
Suppose $d=d_i\in \DN(A)$ is a dominant number for some $i>0$.
A simple module $S$ is in $\Gr_d(A)$ if and only if $\Ext_A^k(S,A)=0$ for all $0 \leq k < d$, which is equivalent to $\Hom_A(S,I^k)=0$ for all $0 \leq k < d$. Thus $S \in \Gr_d(A)$ if and only if $S$ is not a submodule of $I^{<d}:= \bigoplus_{k< d}{I^{k}}$.  For $k\notin \{d_1,\ldots, d_i=d\}$, as $I^k$ is projective we have $\soc(I^k)\in \Gr_0(A)\setminus \Gr_{d_1}(A)$.  Hence, $S$ is not a submodule of $I^{<d}$ if and only if it is not a submodule of $\bigoplus_{j<i}{I^{d_j}}=D(Af_{i-1})$, which is equivalent to $S \in\mod A/Af_{i-1}A$. 
\end{proof}

\subsection{Relation with the Gorenstein condition}
The \emph{Gorenstein condition} is a non-commutative modification of a property of commutative Gorenstein rings and was introduced in the study of Artin-Schelter regular algebras.
The following condition is an analogue of the Gorenstein condition, and was studied under the additional assumption that $S$ has finite projective dimension in \cite{E,MRS} where it is called $k$-regular.

\begin{definition}
A simple module $S$ over an algebra $A$ is called \emph{$k$-Gorenstein} if for $i\geq 0$, $\Ext_A^i(S,A) \neq 0$ if and only if $i=k$ and $\Ext_A^k(S,A)$ is a simple $A^{\op}$-module. We will consider the following condition:
\begin{enumerate}
\item[(G)] $A$ is Iwanaga-Gorenstein and a simple $A$-module $S$ is $k$-Gorenstein for some $k\ge1$ if the injective envelope $I(S)$ of $S$ is not projective. 
\end{enumerate}
\end{definition}

We give the following example, which shows that 0-Gorenstein simple modules correspond to simple algebra direct summands for algebras of finite global dimension.

\begin{example} \label{0-Gorensteinlemma}
Let $A$ be a finite dimensional algebra. Any simple $A$-module corresponding to a simple algebra direct summand of $A$ is 0-Gorenstein. Conversely, if $A$ has finite global dimension, then any $0$-Gorenstein simple $A$-module is obtained in this way.
\end{example}

\begin{proof}
The first assertion is clear. We prove the second one. Without loss of generality, we can assume that $A$ is not simple.
It is basic that $\pdim M= \sup \{ i \geq 0 \mid \Ext_A^i(M,A) \neq 0 \}$ holds for each $M\in\mod A$ with $\pdim M<\infty$.
In particular, each simple module $S$ that is 0-Gorenstein has to be projective since $A$ is assumed to have finite global dimension. Thus we can write $S=eA$ for a primitive idempotent $e\in A$. Since $Ae=\Hom_A(S,A)$ is simple, we have $(1-e)Ae=0=eA(1-e)$.
Thus $S$ is given by a simple algebra direct summand $eAe$.
\end{proof}

\begin{lemma} \label{from P to S}
Assume that $A$ satisfies the condition (G). For every indecomposable projective non-injective $A$-module $P$, there exists a unique simple $A$-module $S$ with injective envelope $I(S)$ non-projective and $\Ext^k_A(S,P)\neq0$ for some $k\ge1$. In particular, the minimal injective coresolution of $P$ is of the form
\begin{align}\label{eq:cond G resolution}
0 \to P\to I^0 \to I^1 \to \cdots \to I^{k-1}\to I(S) \to 0
\end{align}
with $I^0, I^1, \ldots, I^{k-1}$ projective-injective and this defines a bijection $\pi:\ind(\proj A)\to \ind(\inj A)$ given by $P\mapsto \Omega^{-\idim P}(P)$.
\end{lemma}

\begin{proof}
We consider a map
\begin{align*}
\alpha:{\mathcal S}:=\{S\mid\mbox{simple $A$-module such that $I(S)\notin\proj A$}\}\to{\mathcal P}:=\ind(\proj A)\setminus\ind(\inj A)
\end{align*}
defined as follows: By the condition (G), each $S\in{\mathcal S}$ is $k$-Gorenstein for some $k\ge1$. Then $\Ext^k_A(S,A)$ is a simple $A^\op$-module, so there exists a unique $P\in\ind(\proj A)$ such that $\Ext^k_A(S,P)\neq0$. Clearly $P\in{\mathcal P}$. Now we define $\alpha(S):=P$.

We prove that the map is bijective. Since $\#{\mathcal S}=\#{\mathcal P}$, it suffices to show that the map $\alpha$ is surjective.
For each $P\in{\mathcal P}$, let
\begin{equation*}
0\to P\to I^0\to\cdots\to I^k \to0
\end{equation*}
be a minimal injective coresolution. Clearly $P\in\mathcal{P}$ means that $k\ge1$ holds. Take a simple submodule $S$ of $I^k$.
Since $I^k$ has no projective direct summands, we have $S\in{\mathcal S}$.  By the condition (G), $S$ is $k$-Gorenstein, and we have $\alpha(S)=P$.  Thus $\alpha$ is surjective.  As a consequence of $\alpha$ being bijective, we get that $I^k=I(S)$.

Finally, there is a canonical bijection $\beta:\mathcal{S}\to \ind(\inj A)\setminus \ind(\proj A)$ given by $S\mapsto I(S)$. Thus the restriction of the map $\pi$ to $\ind(\proj A)\setminus\ind(\inj A)$ coincides with $\beta\alpha^{-1}$.
\end{proof}

\begin{theorem} \label{conditionGtheorem}
Let $A$ be a finite dimensional algebra. Then $A$ satisfies the condition (G) if and only if $A$ is dominant Auslander-Gorenstein.
\end{theorem}

\begin{proof}
We prove the ``if'' part first. Assume tat $A$ is dominant Auslander-Gorenstein.  To show that $A$ satisfies the condition (G), we show that a simple $A$-module $S$ with a non-projective injective envelop $I$ is $d$-Gorenstein for $d:=\pdim I$.

By Proposition \ref{strongAGpropo}(1), we immediately have $\Ext_A^i(S,P)=0$ whenever $(i,P)\neq (d,\Omega^{d}(I))$ and that the $A^\op$-module $T:=\Ext_A^d(S,A)$ has a unique composition factor.
Since $T\simeq \RHom_A(S,A)[d]$ and $\RHom_A(-,A)$ gives a duality $\D(A)\to\D(A^{\op})$ of the derived categories, the ring $\End_{A^{\op}}(T)\simeq \End_A(S)$ is a division ring. Thus $T$ is a simple $A^{\op}$-module, and hence $S$ is $d$-Gorenstein.

We now prove the ``only if'' part. Assume that $A$ satisfies the condition (G). For any $P\in\ind(\proj A)$, Lemma \ref{from P to S} says that $\idim P\leq \domdim P$. 
\end{proof}

Let $\mathsf{sim}A$ be the set of (isoclasses of) simple $A$-modules.
For an Auslander-Gorenstein algebra $A$, there is a classical permutation 
\[\mathsf{sim}A\to \mathsf{sim} A\ \mbox{ given by }\ S\mapsto \top D\Ext_A^{\gr S}(S,A),\]
which is called \emph{grade bijection} in \cite{I2}.
If $A$ is dominant Auslander-regular, then each simple $A$-module $S$ satisfies $\Ext^i_A(S,A)=0$ for all $i\neq\gr S$, and hence $D\Ext_A^{\gr S}(S,A)$ coincides with the higher Auslander-Reiten translation $\tau_{\gr S}S$.
This observation relates the grade bijection with the bijection $\pi$ in Proposition \ref{strongAGpropo} (2), which is also the same as the $\pi$ in Lemma \ref{from P to S}.

\begin{corollary}\label{cor:grade bijection}
Let $A$ be a dominant Auslander-Gorenstein algebra.  
For $S\in \mathsf{sim}A$, let $I(S)$ be the injective envelope of $S$.
Then the grade bijection coincides with the map $S\mapsto \top \pi^{-1}(I(S))$.
\end{corollary}
\begin{proof}
By Theorem \ref{conditionGtheorem}, $A$ satisfies the condition (G), and so the claim follows from the construction of $\pi=\beta\alpha^{-1}$ shown in Lemma \ref{from P to S}.
\end{proof}

The next proposition gives an easy construction of dominant Auslander-Gorenstein algebras from known examples by tensoring with self-injective algebras.

\begin{proposition}
Let $A$ be a dominant Auslander-Gorenstein algebra and $B$ be a self-injective algebra over a field $K$. Assume $A$ and $B$ are split algebras over the field $K$. Then the algebra $C:=A \otimes_K B$ is again a dominant Auslander-Gorenstein algebra having the same self-injective and dominant dimension.
\end{proposition}

\begin{proof}
Since $A$ and $B$ are split over $K$, every indecomposable projective non-injective $C$-module is isomorphic to a module of the form $P \otimes_K Q$ where $P$ is an indecomposable projective non-injective $A$-module and $Q$ is an indecomposable projective $B$-module. By \cite[Proposition 10]{ERZ} we have in this case $\idim(P \otimes_K Q)=\idim P+\idim Q=\idim P$, because $B$ is self-injective and thus $Q$ projective-injective.
By an argument as in \cite[Lemma 6]{Mue} we have in general $\domdim P \otimes_K Q= \min\{\domdim P,\domdim Q\}$. Since $B$ is self-injective, $\domdim Q= \infty$ and thus $\domdim P \otimes_K Q =\domdim P=\idim P=\idim P \otimes_K Q $, which shows that $C$ is dominant Auslander-Gorenstein.
\end{proof}

We remark that in general the tensor product for two dominant Auslander-Gorenstein algebras is not dominant Auslander-Gorenstein as easy examples such as $KA_2\otimes_K KA_2$ show, where $KA_2$ is the path algebra of the $A_2$-quiver.

\section{Mixed precluster tilting modules and dominant Auslander-Gorenstein algebras }
\label{sec:MPCT}

\subsection{Dominant Auslander-Solberg correspondence}
We generalise the correspondence between minimal Auslander-Gorenstein algebras and precluster tilting modules to the dominant setting.

\begin{definition-proposition}\label{mixed 0}
Let $B$ be a finite dimensional algebra. We say that $M\in\mod B$ is \emph{mixed precluster tilting} if $B\oplus DB\in\add M$ and the following equivalent conditions are satisfied.
\begin{enumerate}[\rm(i)]
\item For each $X\in\ind(\add M)\setminus\ind(\proj B)$, there exists $\ell_X\ge1$ such that $\Ext^i_B(X,M)=0$ for all $1\le i<\ell_X$ and $\tau_{\ell_X}(X)\in\add M$.
\item For each $Y\in\ind(\add M)\setminus\ind(\inj B)$, there exists $r_Y\ge1$ such that $\Ext^i_B(M,Y)=0$ for all $1\le i<r_Y$ and $\tau_{r_Y}^-(Y)\in\add M$.
\end{enumerate}
Moreover, we have $r_Y=\ell_X$ for $Y=\tau_{\ell_X}(X)$.
\end{definition-proposition}

\begin{proof}
We only show (i)$\Rightarrow$(ii); the other direction can be shown dually.

For each $\ell\ge 1$, we have an equivalence $\tau_\ell:\underline{{}^{\perp_{\ell-1}}A}\simeq\overline{DA^{\perp_{\ell-1}}}$ from \eqref{eq:stable equiv perpA DAperp}.  Hence, the map $\ind(\add M)\setminus\ind(\proj B)\simeq\ind(\add M)\setminus\ind(\inj B)$ given by $X\mapsto\tau_{\ell_X}(X)$ is injective.  This has to be bijective since the cardinalities of the domain and the target are the same.
This implies that, for each $Y\in\ind(\add M)\setminus\ind(\inj B)$, there exists $X\in\ind(\add M)\setminus\ind(\inj B)$ such that $Y=\tau_{\ell_X}(X)$.  Using \eqref{eq:stable equiv perpA DAperp} again, we have $\tau^-_{\ell_X}(Y)\simeq X\in\add M$, which then yields
\[\Ext^i_B(M,Y)\simeq D\Ext_B^{\ell_X-i}(X,M)=0\]
for each $1\le i<\ell_X$ by higher AR duality \eqref{eq:AR dual}.  Thus, $r_Y:=\ell_X$ satisfies the desired conditions.
\end{proof}

Now we prove the following generalisation of the Auslander-Solberg correspondence \cite{AS,IyaSol}.

\begin{theorem}\label{main correspondence 0}
There exists a bijection between the following objects.
\begin{enumerate}[\rm(1)]
\item The Morita equivalence classes of dominant Auslander-Gorenstein algebras $A$ with $\domdim A\ge 2$.
\item The Morita equivalence classes of pairs $(B,M)$ of finite dimensional algebras $B$ and mixed precluster tilting $B$-modules $M$.
\end{enumerate}
The correspondence from {\rm (2)} to {\rm (1)} is given by $(B,M)\mapsto A:=\End_B(M)$.
\end{theorem}

We will prove the theorem in two parts: Propositions \ref{B to A 0} and \ref{A to B 0}, but first, we need the following.

\begin{lemma}\label{lem:Hom(M,inj cores)}
Let $B$ be a finite dimensional algebra and $M\in \mod B$ a generator-cogenerator.
If $X\in \add M$ is indecomposable non-injective whose minimal injective coresolution starts with
\begin{equation}\label{resolution of X 0}
0\to X\to I^0 \to I^1 \to \cdots \to I^{d-1}\xrightarrow{a} I^{d},
\end{equation} 
and $X\in M^{\perp d-1}$,
then we have an exact sequence of $A$-modules
\begin{equation}\label{eq:Hom(M,res of X)}
0\to {}_B(M,X) \to {}_B(M,I^0) \to \cdots \to {}_B(M,I^{d-1}) \xrightarrow{(M,a)} {}_B(M,I^d) \to D\Hom_B(\tau_d^-(X),M)\to 0,
\end{equation}
where ${}_B(M,-):=\Hom_B(M,-)$.
\end{lemma}

\begin{proof}
Applying $\Hom_B(M,-)$ and using $\Ext^i_B(M,X)=0$ for all $1\le i<d$, we obtain an exact sequence of the form \eqref{eq:Hom(M,res of X)} except the last term is just $\Cokernel(M,a)$ instead of $D\Hom_B(\tau_d^-(X),M)$.
The functorial isomorphism $\Hom(M,-)\simeq D\Hom_B(\nu^{-1}(-),M)$ on $\inj B$ yields the following commutative diagram 
\[\xymatrix@R1em@C6.5em{
\Hom_B(M,I^{r_X-1})\ar@{=}[d]\ar[r]^{(M,a)}&\Hom_B(M,I^{d})\ar@{=}[d]\\
D\Hom_B(\nu^{-1}(I^{d-1}),M)\ar[r]^{D\Hom_B(\nu^{-1}(a),M)}&D\Hom_B(\nu^{-1}(I^{d}),M),}\]
which implies that $\Cokernel(M,a)=\Cokernel D\Hom_B(\nu^{-1}(a),M)$.

By \eqref{eq:tau=derived nu}, applying $\nu^{-1}$ to the sequence \eqref{resolution of X 0} yields an exact sequence
\[
\nu^{-1}(I^{d-1})\xrightarrow{\nu^{-1}(a)}\nu^{-1}(I^{d})\to\tau_{d}^-(X)\to0.
\]
Applying $D\Hom_B(-,M)$ to this, we obtain an exact sequence
\[D\Hom_B(\nu^{-1}(I^{r_X-1}),M)\xrightarrow{D(\nu^{-1}(a),M)} D\Hom_B(\nu^{-1}(I^{r_X}),M)\to D\Hom_B(\tau_{r_X}^-(X),M)\to0.\]
Hence, we have $D\Hom_B(\tau_{r_X}^-(X),M)=\Cokernel D\Hom_B(\nu^{-1}(a),M) = \Cokernel (M,a)$, as required.
\end{proof}

\begin{proposition}\label{B to A 0}
Let $B$ be a finite dimensional algebra.
If $M$ is a mixed precluster tilting $B$-module, then $A=\End_B(M)$ is a dominant Auslander-Gorenstein algebra with $\domdim A \ge2$.
Moreover, we have 
\[
\idim \Hom_B(M,X)=r_X+1 \text{ and }\pdim D\Hom_B(X,M) = \ell_X+1
\]
for each non-injective indecomposable $X\in \add M$; whereas $\Hom_B(M,X)\simeq D\Hom_B(\nu^{-1}(X),M)$ is projective-injective for each injective indecomposable $X\in\inj A$.
\end{proposition}

\begin{proof}
Let $A=\End_B(M)$. Fix an indecomposable projective $A$-module $P$. Then we have $P=\Hom_B(M,X)$ for some indecomposable $X\in\add M$.

(i) Assume $X\in\inj B$. Then $P=\Hom_B(M,X)\simeq D\Hom_B(\nu^{-1}(X),M)$ is an injective $A$-module.  Hence, we have $0=\idim P \le \domdim P=\infty$.

(ii) Assume $X\notin\inj B$ and consider the indecomposable non-injective $P:=\Hom_B(M,X)\in \proj A$. 
By Lemma \ref{lem:Hom(M,inj cores)}, we have an exact sequence of $A$-modules
\[0\to P\to\Hom_B(M,I^0)\to\cdots\to\Hom_B(M,I^{r_X-1})\xrightarrow{(M,a)}\Hom_B(M,I^{r_X})\to D\Hom_B(\tau_{r_X}^-(X),M)\to0,
\]
where $0\to X\to I^0\to \cdots \to I^{r_X}$ is the first $r_X+1$ terms of the minimal injective coresolution of $X$.
The terms $\Hom_B(M,I^i)\in \proj A\cap \inj A$ for all $0\leq i\leq r_X$ as $I^i\in \inj B$.
Since $Y:=\tau_{r_X}^-(X)\in \add(M)$ by our setup, the last term in the above sequence is an injective $A$-module.
In particular, we have $\domdim P = \idim P = r_X+1$ and $\pdim I =\ell_Y+1$ for the non-projective injective $I=D\Hom_B(Y,M)$.  This implies the dominant Auslander-Gorenstein property.
\end{proof}

The converse map of Theorem \ref{main correspondence 0} is given by the following observation.

\begin{proposition}\label{A to B 0}
Let $A$ be a dominant Auslander-Gorenstein algebra with $\domdim A\ge2$. 
Let $B=fAf$ be the base algebra associated to the idempotent $f\in A$ with $\proj A\cap\inj A=\add D(Af)$.
Then $M:=Af$ is a mixed precluster tilting $B$-module and satisfies $\End_B(M)=A$.
\end{proposition}


\begin{proof}
Since $\domdim A\ge 2$, it follows from Proposition \ref{prop:DCP}(1)(2) that $M=Af$ is a generator-cogenerator of $B$ satisfying $A\simeq \End_B(Af)$.  It remains to show that $M$ is a mixed precluster tilting $B$-module.

Fix $X\in\ind(\add M)\setminus\ind(\inj B)$.  
Since $M=Af$, we have some indecomposable projective $P\in \proj A$ so that $X=Pf=\Hom_A(fA,P)$.
Being dominant Auslander-Gorenstein with $\domdim A\geq 2$, we have a minimal injective coresolution of $P$ of the form
\begin{equation}\label{eq:proj res P in proj A}
0\to P\to Q^0 \to Q^1\to \cdots \to Q^d \to Q^{d+1}\to 0
\end{equation}
for some $d\ge 1$ and $Q^i$ projective-injective for all $0\le i\le d$, and $Q^{d+1}$ is an indecomposable non-projective injective $A$-module.
Our aim is to show that $X\in M^{\perp d-1}$ and $\tau_d^-(X)\in \add M$, i.e.\ $r_X:=d$ satisfies the condition (i) in Definition-Proposition \ref{mixed 0}.

Since $\domdim P=d+1$, it follows from Proposition \ref{APTpropo}(3) that $\Ext_B^i(M,X)=0$ for all $1\leq i<d$.
By Lemma \ref{lem:Hom(M,inj cores)}, the first $d+1$ terms of the minimal injective coresolution \eqref{resolution of X 0} of $X$ induces an exact sequence of $A$-modules of the form \eqref{eq:Hom(M,res of X)}.
By minimality, the first $d+1$ terms
\[0\to P\to \Hom_B(M,I^0) \to \cdots \to \Hom_B(M,I^d)\]
coincides with those in \eqref{eq:proj res P in proj A} up to $Q^d$. 
By exactness, the sequence \eqref{eq:Hom(M,res of X)} is the same as \eqref{eq:proj res P in proj A}.
In particular, the last term $D\Hom_B(\tau_d^-(X),M)$ is an injective $A$-module $Q^{d+1}$, and hence $\tau_d^-(X)\in\add M$ holds by Proposition \ref{prop:DCP}(3).
\end{proof}

\subsection{Mixed cluster tilting modules and Dominant Auslander correspondence}

In this subsection, we give a description of the category $\ZZ(M)$ when $A$ is a dominant Auslander-Gorenstein algebra with dominant dimension at least two that generalises the result from \cite{IyaSol} where such a description was given for minimal Auslander-Gorenstein algebras. We apply this description to give characterisation when dominant Auslander-Gorenstein algebras with dominant dimension at least two have finite global dimension and use this to generalise the classical higher Auslander correspondence.

\begin{definition-proposition} \label{ZMdefinition}
Let $B$ be an algebra, and $M$ a mixed precluster tilting $B$-module. We consider a full subcategory
\[\ZZ(M):=\bigcap\limits_{X\in\ind(\add M)}X^{\perp{\ell_X-1}}=\bigcap\limits_{X\in\ind(\add M)}{^{\perp{r_X-1}}X}.\]
We call $M$ a \emph{mixed cluster tilting} $B$-module if $\ZZ(M)=\add M$.
\end{definition-proposition}

The equality 
$$\bigcap\limits_{X\in\ind(\add M)}X^{\perp{\ell_X-1}}=\bigcap\limits_{X\in\ind(\add M)}{^{\perp{r_X-1}}X}$$
for mixed precluster tilting modules $M$ follows as in Proposition \ref{mixed 0} and its proof.
As in the case of precluster tilting modules, we have the following equivalence.


\begin{theorem}\label{thm-CM}
Let $A$ be a dominant Auslander-Gorenstein algebra with dominant dimension at least two corresponding to $(B,M)$.
Then we have a commutative diagram of equivalences of categories:
\[
\xymatrix@C=12em@R=2em{ \ZZ(M)\ar@{=}[d] \ar@<.5ex>[r]^{\Hom_B(M,-)}&\CM A \ar@<.5ex>[d]^{\Hom_A(-,A)}\ar@<.5ex>[l]^{-\otimes_A M}\\ 
\ZZ(M) \ar@<.5ex>[r]^{\Hom_B(-,M)}&(\CM A^{\op})^{\op}.\ar@<.5ex>[u]^{\Hom_{A^{\op}}(-,A)}\ar@<.5ex>[l]^{\Hom_{A^{\op}}(-,M)} }
\]
\end{theorem}

\begin{proof}
For an Iwanaga-Gorenstein algebra $A$ with $\domdim A\ge 2$ and self-injective dimension $g$, we have $\CM A=\Omega^g(\mod A) \subseteq \Omega^2(\mod A)=\Dom_2(A)$, where the first equality is well-known (e.g.\ \cite{Che}) and the last equality follows from Lemma \ref{marvillemma}. Hence, every maximal Cohen-Macaulay $A$-module has dominant dimension at least $2$.  By restricting the equivalence $-\otimes_A Af:\Dom_2(A)\simeq \mod fAf$ from Proposition \ref{APTpropo} (1), we obtain an equivalence of categories $\CM A\simeq \CC:=F(\CM A)=\{Uf\mid U\in\CM A\}$ where $F:=-\otimes_A Af$.
We want to show that $\CC=\ZZ(M)$ when $A$ is dominant Auslander-Gorenstein with corresponding mixed precluster tilting pair $(B,M)$.

Let $P_X:=\Hom_B(M,X)$ be the projective $A$-module corresponding to an indecomposable $X\in \add M$.
Recall from Proposition \ref{prop:DCP}(3) that the projective non-injective $P_X$ are given by $X\in \add M\setminus \inj B$.
Let us now show that $\CC\subset \ZZ(M)$.
Recall that $\CM A := {}^{\perp g}A$, so $U\in {}^{\perp g}A$ if and only if $U\in {}^{\perp \idim P_X}P_X$ for every projective non-injective modules $P_X=\Hom_B(M,X)$.  
By Proposition \ref{A to B 0}, we have $\idim P_X = r_X+1$ is the number associated to $X$ in the definition of mixed precluster tilting.
Since $\domdim P_X = \idim P_X = r_X+1$, it follows from Proposition \ref{APTpropo} (2) that $\Ext_A^{r}(U,P_X) \simeq \Ext_{fAf}^r(Uf,P_X f) \simeq \Ext_{fAf}^r(Uf,X)$ for all $1\leq r< r_X$. 
Thus $U \in \CM A$ implies that $Uf \in \bigcap_{X\in\ind(\add M)}{^{\perp r_X-1}X}=\ZZ(M)$. 

Let $U:=\Hom_B(M,Z)\in \mod A$ for $Z\in \ZZ(M)$.
We want to show that $U\in \CM A = \bigcap_{X\in \ind(\add M)}{}^{\perp >0} P_X$.
It is clear that $U\in {}^{\perp >0} P_X$ for projective-injective $P_X$ (equivalently, $X\in \inj B$); so it remains to consider the case $X\in \add M\setminus \inj B$.
Since $Z=Uf$, using Proposition \ref{A to B 0} and Proposition \ref{APTpropo} (2) again, we get that 
$\Ext_A^i(U,P_X)\simeq \Ext_B^i(Z,X)=0$ for all $1\leq i<r_X$, i.e.
$U\in \bigcap_{X\in\ind(\add M)}{}^{\perp r_X-1}P_X$. 
Since $U \in \Dom_2(A)\subset\Omega^2(\mod A)$, we can write $U=\Omega^2(U')$ for some $A$-module $U'$.
Then, for $t \geq r_X$, we have $\Ext_A^t(U,P_X)=\Ext_A^t(\Omega^2(U'),P_X)=\Ext_A^{t+2}(U',P_X)=0$ since $\idim P_X=r_X+1 \leq t+2$. 
Thus, we have $U\in {}^{\perp >0} P_X$ for all $P_X\in \ind(\proj A)$, which means that $U \in \CM A$.

To see the commutativity of the diagram, first note that $\Hom_A(-,A)$ gives an equivalence between $\CM A$ and $(\CM A^{\op})^{\op}$, see for example corollary 2.16 in \cite{Che}.
We only prove that $-\otimes_AM:\CM A\to\ZZ(M)$ is an equivalence. Then its right adjoint functor $\Hom_B(M,-):\ZZ(M)\to\CM A$ gives its quasi-inverse.
Since $M$ is a projective $A^\op$-module, we have an isomorphism of functors
\[\Hom_{A^\op}(-,M)\simeq\Hom_{A^\op}(-,A)\otimes_AM=(-\otimes_AM)\circ\Hom_{A^\op}(-,A).\]
Since $M$ is a generator, $\Hom_A(M,-):\ZZ(M)\to\CM A$ is fully faithful and we have an isomorphism of functors
\begin{align*}
\Hom_B(-,M)\simeq\Hom_A(\Hom_B(M,-),\Hom_B(M,M))=\Hom_A(-,A)\circ\Hom_B(M,-).
\end{align*}
Thus the commutativity of the diagram follows. In particular, the lower horizontal functors are equivalences, and all claims follow.
\end{proof}

We obtain the next result immediately.

\begin{corollary}\label{thm-CM2}
Let $A$ be a dominant Auslander-Gorenstein algebra with dominant dimension at least two associated to $(B,M)$. Then the following are equivalent:
\begin{enumerate}[\rm(1)]
\item $\gldim A < \infty$.
\item $\CM A=\proj A$.
\item $M$ is a mixed cluster tilting $B$-module.
\end{enumerate}
In such a case, $\gldim A$ coincides with the self-injective dimension of $A$.
\end{corollary}
\begin{proof}
The equivalence of (2) and (3) follows directly from the equivalence $\proj A\simeq \add M$ and Theorem \ref{thm-CM} which asserts that $\CM A\simeq \ZZ(M)$.
The equivalence of (1) and (2) follows immediately from the fact that an Iwanga-Gorenstein algebra has finite global dimension if and only if $\CM A=\proj A$ and in this case the global dimension equals the self-injective dimension of $A$.
\end{proof}

As a consequence of Theorem \ref{main correspondence 0} and Corollary \ref{thm-CM2}, we obtain a generalisation of the classical higher Auslander correspondence. Due to its importance we formulate it as a theorem.

\begin{theorem}\label{main correspondence 0 2}
There exists a bijection between the following objects.
\begin{enumerate}[\rm(1)]
\item The Morita equivalence classes of dominant Auslander-regular algebras $A$ with $\domdim A\ge 2$.
\item The Morita equivalence classes of pairs $(B,M)$ of finite dimensional algebras $B$ and mixed cluster tilting $B$-modules $M$.
\end{enumerate}
The correspondence from (2) to (1) is given by $(B,M)\mapsto A:=\End_B(M)$.
\end{theorem}

The following is a well-known useful lemma.

\begin{lemma}\label{lem:addM resol to proj res} {\rm \cite{A}\cite[Lemma 2.1]{EHIS}}
Let $B$ be an algebra with a generator-cogenerator $M$ such that $\gldim \End_B(M)\leq n$ for some $n\geq 2$.
Then, for every $B$-module $Y$, there exists a right $(\add M)$-resolution
$$0 \rightarrow M_{n-2} \rightarrow \cdots \rightarrow M_1 \rightarrow M_0 \rightarrow Y \rightarrow 0,$$
that is, an exact sequence with all $M_i$'s in $\add M$, such that $\Hom_B(M,-)$ is exact on this sequence.
\end{lemma}

Next, we show a generalisation of the main result in \cite{EH} which relates the existence of mixed cluster tilting modules with the complexity of a self-injective algebra.  This also gives a shorter proof for the main result of \cite{EH}.
Recall the notion of complexity of modules and algebras from \eqref{eq:cx(M)}.

\begin{theorem} \label{ErdmannHolmgeneralisation}
Let $B$ be a self-injective algebra and $M$ a mixed precluster tilting module in $\mod B$.
\begin{enumerate}[\rm(1)]
\item The complexity of $M$ is at most $1$, that is, the $K$-dimension of the terms in the minimal projective resolution of $M$ is bounded.
\item If $M$ is, moreover, a mixed cluster tilting module, then the complexity of $B$ is at most $1$, that is, the complexity of every $B$-module is bounded by $1$.
\end{enumerate}
\end{theorem}
\begin{proof}
(1) Fix any $X\in\ind(\add M)\setminus\ind(\proj B)$.
Since $B$ is self-injective, we have $\tau=\nu \Omega^{2}$ and hence $ \nu\Omega^{\ell_X +1}(X)=\tau_{\ell_X}(X)$, which is in $\ind(\add M)\setminus\ind(\proj B)$ by our assumption on $M$. 
Repeating this process, and using the assumption that $\ind(\add M)$ is a finite set, we obtain that $\nu^j\Omega^{i_X}(X)\simeq X$ for some $j, i_X>0$.  Since $\nu$ preserves the vector space dimension, we have $\dim_K \Omega^{i_X}(X)=\dim_K X$ and so the $K$-dimension of the minimal projective resolution of $X$ is bounded by those from the first ${i_X}+1$ terms.  Thus, the $K$-dimension of terms in the minimal projective resolution of $M$ is bounded by the first $i$ terms for $i=\max\{i_X+1\mid X\in \ind(\add M)\setminus \ind(\proj B)\}$.

(2) By Lemma \ref{lem:addM resol to proj res}, every $B$-module $N$ admits a right $(\add M)$-resolution 
$$0 \rightarrow M_n \rightarrow M_{n-1} \rightarrow \cdots \rightarrow M_1 \rightarrow M_0 \rightarrow N \rightarrow 0.$$
Since the subcategory of $B$-modules of complexity at most $1$ is closed under cokernels of monomorphisms, the above sequence and (1) imply that $N$ has complexity at most $1$. 
\end{proof}

We remark that the theorem can fail if $B$ is not self-injective, see \cite{MV} for a counter-example.

We give an example that shows that, in contrast to the cluster tilting situation \cite[Theorem 2.3(1)]{I5}\cite[Proposition 3.8(b)]{IyaSol}, the conditions $\tau_{r_X}^-(X)\in\add M$ and $\tau_{\ell_X}(X)\in \add M$ do not follow from the rest of the conditions in the definition of a mixed cluster tilting module.

\begin{example}
Let $A$ be the connected symmetric Nakayama algebra of Loewy length $3$ with two simple modules $S$ and $T$.
Let $X$ be indecomposable $A$-module of length two with simple top $S$ and simple socle $T$.
Take $M:=A \oplus S \oplus X$, $r_{S}=1=\ell_{X}$ and $r_{X}=2=\ell_{S}$.
It is routine to check that $S \in {M}^{\perp \ell_S-1}$ and $X\in {}^{\perp r_X-1}M$, and ${}^{\perp 1} X = \add M = S^{\perp 1}$.
However, $M$ is not mixed cluster tilting; indeed, $B:=\End_A(M)$ is given by the
\[
\text{quiver }\vcenter{\xymatrix{
P_S \ar[r]^{a}  \ar[d]_{c} & P_T \ar[d]^{b} \\
S \ar[r]_{d} & X \ar[ul]_{e}
}} \;\;\text{ with relations }ab-cd, dea, ec.
\]
The indecomposable projective $B$-module corresponding to the point $P_S$ in the quiver has dominant dimension 2 and injective dimension 3.  Hence, $B$ is not a dominant Auslander-regular algebra and thus $M$ cannot be a mixed cluster tilting. Note that $\tau_2(X) = T\notin \add M$ and $\tau_2^-(S)=\rad P_S\notin \add M$.
\end{example}

\section{Mixed cluster tilting for trivial extensions of $d$-representation finite algebras} \label{section: TrivExtdrep}

Recall that an algebra $A$ is called \emph{$d$-representation-finite} if $A$ has global dimension at most $d$ and a $d$-cluster tilting module. 
The \emph{trivial extension} $T(A)$ of a finite dimensional $K$-algebra $A$ is defined as the algebra with vector space structure $A \oplus DA$ and multiplication $(a,f)(b,g)=(ab,ag+fb)$, where $DA=\Hom_K(A,K)$.
The trivial extension algebra $T(A)$ is a symmetric algebra, that is $A \simeq DA$ as $A$-bimodules. 
We will often use in this section that, for self-injective algebras $A$, we have $\Ext_A^i(M,N) \simeq \underline{\Hom}_A(\Omega_A^i(M),N)$ for all $i\ge 0$ and $\tau_d(M) \simeq \nu \Omega^{d+1}(M)$ for all $d\ge 1$.
Note that an algebra is symmetric exactly when the Nakayama functor is isomorphic to the identity functor.
We refer for example to \cite{SkoYam} for more information on trivial extension algebras and symmetric algebras.

Another fact we will use about $T(A)$ is that $T(A)$ can be $\Z$-graded with $\deg A=0$ and $\deg DA=1$.  Moreover, we have an equivalence of triangulated categories $\underline{\mod}^{\Z} T(A)\simeq \Db(\mod A)$ between the stable category of $\Z$-graded finitely generated $T(A)$-module and the bounded derived category of $\mod A$ \cite{Hap}.  Under this equivalence, the autoequivalence $\Omega_{T(A)}\circ(1)$ given by composing the syzygy with the grading-shift corresponds to the derived Nakayama functor $\nu_A$; see \cite[Section 2.2, 2.3]{CDIM} for the details.

The following theorem shows the existence of a canonical mixed cluster tilting module for the trivial extension algebra of a $d$-representation-finite algebra.

\begin{theorem}\label{B+N}
For $d\ge1$, let $A$ be a $d$-representation-finite algebra with basic $d$-cluster tilting $A$-module $A\oplus N$, and let $B=T(A)$ be the trivial extension of $A$.
Then the $B$-module $B \oplus N$ is a mixed cluster tilting $B$-module.
In this case, for each $X\in\ind(\add N)$, $\ell_X$ is given by
\[\ell_X=d+t_X\ \mbox{ where $t_X=\min \{ l \geq 0 \mid \nu_A^l(\tau_d^A(X))$ is not projective$\}$.}\]
\end{theorem}

We start with recalling a basic result of $d$-representation-finite algebras.

\begin{proposition}\cite[Lemma 2.3]{I7}\label{basic of d-CT}
Let $A$ be a $d$-representation-finite algebra with basic $d$-cluster tilting $A$-module $M=A\oplus N$. Then we have a bijection $\tau_d:\ind(\add M)\setminus \ind(\proj A)\simeq\ind(\add M)\setminus\ind(\inj A)$, and we have $\Hom_A(N,A)=0$.
\end{proposition}

To prove the previous theorem, we need the following two lemmas:

\begin{lemma} \label{perfectmodultautrivextlemma}
Let $A$ be an algebra with trivial extension $B=T(A)$, and take $X\in\mod A$ satisfying
$\Ext_A^i(X,A)=0$ for each $0\le i<d:=\pdim X_A$.  Then the following hold.
\begin{enumerate}[\rm(1)]
\item $\tau_d^A(X) \simeq \tau_d^B(X)$.
\item Take a minimal projective resolution of $X$ as an $A$-module:
\begin{align}\label{mpr X}
0 \rightarrow P_d \xrightarrow{f_d} P_{d-1} \rightarrow \cdots \rightarrow P_0 \rightarrow X \rightarrow 0.
\end{align}
Applying $-\otimes_AB$, we obtain an exact sequence of $B$-modules:
\begin{align}\label{mpr X 2}
0\to\tau_d^A(X)\to P_d \otimes_A B \xrightarrow{} P_{d-1} \otimes_A B \rightarrow \cdots \rightarrow P_0 \otimes_A B \rightarrow X \rightarrow 0.
\end{align}
\end{enumerate}
\end{lemma}
\begin{proof}
Since $\tau_d^B(X) =\Omega_B^{d+1}(X)$, it suffices to prove (2).
Applying $-\otimes_AB$ to \eqref{mpr X}, we get the start of a projective resolution
\begin{align}\label{mpr X 1.5}
P_d \otimes_A B \xrightarrow{f_d \otimes_A \id_B} P_{d-1} \otimes_A B \rightarrow \cdots \rightarrow P_0 \otimes_A B \rightarrow X \otimes_A B \rightarrow 0.
\end{align}
This is exact since $B \simeq A \oplus DA$ as $A$-modules and hence $\Tor_i^A(X,B)=\Tor_i^A(X,DA) \simeq D \Ext_A^i(X,A)=0$ holds for $1 \le i <d$, which means that $X\otimes_AB\simeq X$ holds as $X \otimes_A DA \simeq D \Hom_A(X,A) =0$.

Now we claim that the kernel of the map $f_d \otimes_A \id_B$ is isomorphic to $\tau_d^A(X)$ as $B$-modules.
Since $P_i \otimes_A B \simeq P_i \oplus (P_i \otimes_A DA)$ as $A$-modules, the sequence \eqref{mpr X 1.5} regarded as a sequence of $A$-modules is isomorphic to the direct sum of \eqref{mpr X} and the
exact sequence
$$0 \rightarrow \tau_d^A(X) \rightarrow P_n \otimes_A DA \rightarrow \cdots \rightarrow P_0 \otimes_A DA \rightarrow X \otimes_A DA \rightarrow 0$$
obtained by applying $-\otimes_ADA$ to \eqref{mpr X}. Consequently, \eqref{mpr X 2} is exact.
\end{proof}

The following observation is clear.

\begin{lemma}
For each $P\in\ind(\proj A)$, let $t=\min \{ l \geq 0 \mid \nu_A^l(P)$ is not projective$\}$. Then combining the projective covers
\begin{align*}
0\to\nu^{i+1}(P)\to\nu^i(P)\otimes_AB\to\nu^i(P)\to 0
\end{align*}
for each $0\le i< t$, we obtain an exact sequence
\begin{align}\label{left of mpr of X}
0\to \nu^t(P)\to \nu^{t-1}(P)\otimes_AB\to\cdots\to \nu(P)\otimes_AB\to P\otimes_AB\to P \rightarrow 0.
\end{align}
\end{lemma}


We will prove Theorem \ref{B+N} in two steps.  First, we show in the following lemma that $B\oplus N$ is mixed precluster tilting with the claimed value of $\ell_X$.  Afterwards we show that the arising dominant Auslander-Gorenstein has finite global dimension.

\begin{lemma}\label{B+N mpct}
Under the setting of Theorem \ref{B+N}, $B\oplus N$ is a mixed precluster tilting $B$-module with $\ell_X=d+t_X$
for all $X\in \ind(\add N)$.
\end{lemma}

\begin{proof}
We can assume for simplicity that $A$ and $N$ are basic.
By Proposition \ref{basic of d-CT} , we have $\Ext_A^i(N,A)=0$ for all $0
\le i<d$.
We now calculate $\ell_X$ for $X\in \ind(\add M)$; we do this for two separate cases:

(i) Let $X\in\ind(\add N)$ such that $\tau_d(X)$ is non-projective.  We will show that $\ell_X=d$, i.e. $\tau_d^B(X)\in \add N$ and $X\in {}^{\perp d-1}(B\oplus N)$.
The first required property already follows from Lemma \ref{perfectmodultautrivextlemma} which asserts that $\tau_d^B(X)\simeq \tau_d^A(X)\in \add N$, so it remains to show orthogonality.

Let \eqref{mpr X} be a minimal projective resolution of $X$.
Applying $\Hom_A(-,N)$ to \eqref{mpr X} and $\Hom_B(-,N)$ to \eqref{mpr X 2} and comparing them via a functorial isomorphism $\Hom_B(-\otimes_A B,N) \simeq \Hom_A(-,N)$, we have a commutative diagram of exact sequences
\begin{align}\label{compare}
\xymatrix{
\Hom_A(P_0,N)\ar[r]\ar[d]^\wr&\Hom_A(P_1,N)\ar[r]\ar[d]^\wr&\cdots\ar[r]&\Hom_A(P_d,N)\ar[d]^\wr\\
\Hom_B(P_0\otimes_AB,N)\ar[r]&\Hom_B(P_1\otimes_AB,N)\ar[r]&\cdots\ar[r]&\Hom_B(P_d\otimes_AB,N).
}
\end{align}
Comparing the cohomologies, we obtain that $\Ext^i_B(X,B\oplus N)=\Ext^i_B(X,N)\simeq\Ext^i_A(X,N)=0$ for each $1\le i <d$.


(ii) Let $X\in\ind(\add N)$ such that $P:=\tau_d^A(X)$ is a projective $A$-module and $t:=t_X$.
We have
\[
\tau_{d+t}^B(X)=\Omega_B^{d+t+1}(X)=\Omega_B^{t}(\tau_d^B(X))=\Omega_B^t\tau_d^A(X)=\nu_A^t(\tau_d^A(X)).
\]
Here, the third equality follows from Lemma \ref{perfectmodultautrivextlemma}, and the last equality follows from the correspondence between $\Omega_B\circ(1)$ and $\nu_A$ (and then apply the grading forgetful functor).
Since $\nu_A^{t-1}(\tau_d^AX)\in \proj A$ by assumption, $\tau_{d+t}^B(X)=\nu_A^t(\tau_d^AX) \in \inj A\setminus \proj A \subset \add N$.

For Ext-orthogonality, as in case (i), 
we have $\Ext^i_B(X,N)=0$ for each $1\le i<d$. 
Moreover, since $\Ext^d_A(X,N)\simeq D\overline{\Hom}_A(N,P)=0$ by Proposition \ref{basic of d-CT}, the rightmost map in the upper sequence of \eqref{compare} is surjective, and hence so is the rightmost map in the lower sequence of \eqref{compare}. This implies that $\Ext^d_B(X,N)=0$.

Let $\nu:=\nu_A$.
Gluing \eqref{mpr X 2} and \eqref{left of mpr of X},
we obtain an exact sequence
\[0\to \nu^t(P)\to \nu^{t-1}(P)\otimes_AB\to\cdots\to \nu(P)\otimes_AB\to P\otimes_AB\to P_d \otimes_A B \rightarrow \cdots \rightarrow P_0 \otimes_A B \rightarrow X \rightarrow 0.\]
Thus we have
\[\tau_{d+t}^B(X)=\Omega^{d+t+1}_B(X)=\nu^t(P)\in\add N\ \mbox{ and }\ \Omega_B^i(X)=\nu^{i-d-1}(P)\]
for each $d+1\le i\le d+t-1$. Hence, for such an $i$,  $\Ext^i_B(X,N)=\underline{\Hom}_B(\Omega_B^i(X),N)$ is a quotient of
\[\Hom_B(\nu^{i-d-1}(P),N)=\Hom_A(\nu^{i-d-1}(P),N)=D\Hom_A(N,\nu^{i-d}(P)),\]
which is zero by Proposition \ref{basic of d-CT} as $\nu^{i-d}(P)\in\proj A$. Now the assertion follows.
\end{proof}

Now we are ready to prove Theorem \ref{B+N}.

\begin{proof}[Proof of Theorem \ref{B+N}] 
We have shown in Lemma \ref{B+N mpct} that $B \oplus N$ is mixed precluster tilting, and so it remains to show that the global dimension of $C:=\End_B(B \oplus N)$ is finite, i.e. every  simple $C$-modules $S_X:=\top \Hom_B(B \oplus N,X)$ for indecomposable $X \in \add(B \oplus N)$ have finite projective dimension.
We calculate this in two cases.

(i) First we consider the case $X \in \add N$.

Take a $d$-almost split sequence in $\add (A\oplus N)$
\begin{align}\label{d-ass}
0 \rightarrow \tau_d^A(X) \rightarrow E_{d-1} \rightarrow \cdots\to E_0 \rightarrow X \rightarrow 0.
\end{align}
Since $B$ is a projective $B$-module, applying ${}_B(N,-):=\Hom_B(N,-)$ to \eqref{d-ass} yields an exact sequence
\begin{align*} 
0 \rightarrow  {}_B(N,\tau_d^A(X)) \rightarrow {}_B(N,E_{d-1}) \rightarrow \cdots\to{}_B(N,E_0) \rightarrow {}_B(N,X) \rightarrow  \top{}_B(N,X) \rightarrow 0.
\end{align*}
Now let $t:=t_X$. 
Gluing \eqref{d-ass} and \eqref{left of mpr of X}, we obtain an exact sequence
\begin{align}\label{proj+d-ass}
0\to \nu^t(P)\to \nu^{t-1}(P)\otimes_AB\to\cdots\to \nu(P)\otimes_AB\to P\otimes_AB\to  E_{d-1} \rightarrow \cdots\to E_0 \rightarrow X \rightarrow 0
\end{align}
with terms in $\add(B\oplus N)$. By Proposition \ref{basic of d-CT}, we have $\Hom_A(N,\nu^i(P))=0$ for each $0\le i<t$ and ${}_B(N,\nu^i(P)\otimes_AB)=\Hom_A(N,\nu^{i+1}(P))=0$ for each $0\le i\le t-2$.
Hence, by applying ${}_B(N,-):=\Hom_B(N,-)$ to \eqref{proj+d-ass} we obtain an exact sequence
\begin{align}\label{(N,-)} 
0 \rightarrow  {}_B(N,\nu^t(P))&\to {}_B(N,\nu^{t-1}(P)\otimes_AB)\to 0\to\cdots\to0\\ \notag
&\to{}_B(N,E_{d-1}) \rightarrow \cdots\to{}_B(N,E_0) \rightarrow {}_B(N,X) \rightarrow  \top\Hom_B(N,X) \rightarrow 0.
\end{align}
Applying $\Hom_B(B \oplus N,-)=(-)\oplus\Hom_B(N,-)$ to \eqref{proj+d-ass}, we get a direct sum of \eqref{proj+d-ass} and \eqref{(N,-)}, which gives a projective resolution of the simple $C$-module $S_X$ with $\pdim S_X\leq d+t$.

(ii) Next we consider the case $X \in \proj B$.

Let $P_0$ be an indecomposable projective $A$-module with simple top $S$ and let
$$0 \rightarrow P_m \rightarrow \cdots \rightarrow P_0 \rightarrow S \rightarrow 0$$
be a minimal projective resolution of $S$ as an $A$-module.
Regarding $\Omega_A^iS$ as a submodule of $P_{i-1}$, we define $X_i\in\mod B$ by $X_i:=\left[\begin{array}{c}\Omega_A^iS\\ \nu(P_{i-1})\end{array}\right]\subset P_{i-1}\otimes_AB=\left[\begin{array}{c}P_{i-1}\\ \nu(P_{i-1})\end{array}\right]$. 
Here, we write $X_i$ and $P_{i-1}\otimes_AB$ so that when treated as graded $T(A)$-module the degree 0 and 1 part corresponds to the top and bottom entry in the vector respectively.
Now for each $0\le i\le m$, we have an exact seuquence of $B$-modules
\begin{align}\label{X_i}
0\to X_{i+1}\to(P_i\otimes_AB)\oplus\nu(P_{i-1})\to X_i\to0,
\end{align}
where, when treated as exact sequence of graded modules, the degree 0 part is $0\to\Omega_A^{i+1}S\to P_i\to\Omega_A^iS\to0$ and the degree 1 part is the split exact sequence $0\to\nu(P_i)\to\nu(P_i)\oplus\nu(P_{i-1})\to\nu(P_{i-1})\to0$.

Since $\Hom_B(N,X_i)=\Hom_B(N,P_{i-1}\otimes_AB)=\Hom_A(N,\nu(P_{i-1}))$ holds, applying $\Hom_B(N,-)$ to \eqref{X_i}, we obtain an exact sequence
\[0\to\Hom_B(N,X_{i+1})\to\Hom_B(N,(P_i\otimes_AB)\oplus\nu(P_{i-1}))\to\Hom_B(N,X_i)\to0.\]
Thus applying $\Hom_B(B\oplus N,-)$ to \eqref{X_i}, we obtain an exact sequence
\[0\to\Hom_B(B\oplus N,X_{i+1})\to\Hom_B(B\oplus N,(P_i\otimes_AB)\oplus\nu(P_{i-1}))\to\Hom_B(B\oplus N,X_i)\to0.\]
Now gluing \eqref{X_i} for each $0\le i\le m+1$, we obtain an exact sequence
\[0\to\nu(P_m)\to (P_m\otimes_AB)\oplus\nu(P_{m-1})\to\cdots\to(P_1\otimes_AB)\oplus\nu(P_0)\to P_0\otimes_AB\to S\to0.\]

Finally, applying ${}_B(B\oplus N,-):=\Hom_B(B\oplus N,-)$, we obtain an exact sequence
\begin{align*}
0\to{}_B(B\oplus N,\nu(P_m))&\to{}_B(B\oplus N,(P_m\otimes_AB)\oplus\nu(P_{m-1}))\to\cdots\\
&\to{}_B(B\oplus N,(P_1\otimes_AB)\oplus\nu(P_0))\to{}_B(B\oplus N,P_0\otimes_AB)\to S_X\to0,
\end{align*}
which is a projective resolution of $S_X$.  Hence, we have $\pdim S_X\le m+1\leq d+1$.
\end{proof}

\begin{example}
Let $A$ be the path algebra of the following quiver of Dynkin type $A_3$:
\[\begin{tikzcd}
	1 \ar[r]& 2\ar[r] & 3
\end{tikzcd}\]
The trivial extension $B$ of $A$ has the following quiver:
\[\begin{tikzcd}
	1 \ar[r]& 2\ar[r] & 3 \ar[bend left]{ll}
\end{tikzcd}\]
and the relations are such that all indecomposable projective modules have vector space dimension 4 (so that it is the Nakayama algebra with Kupisch series [4,4,4]).
Let $N=S_1 \oplus S_2 \oplus I_2$ be the direct sum of the indecomposable non-projective $A$-modules with $I_2=D(Ae_2)$ the indecomposable injective.
Then the quiver and relations of $C:=\End_B(B \oplus N)=KQ/I$ is given as follows, 

\[Q: \;\; \begin{tikzcd}
	{P_3} && {P_1} && {S_1} \\
	& {P_2} && {I_2} \\
	&& {S_2}
	\arrow["c", from=1-1, to=1-3]
	\arrow["d", from=1-3, to=1-5]
	\arrow["e", from=1-5, to=2-4]
	\arrow["f", from=2-4, to=3-3]
	\arrow["a", from=3-3, to=2-2]
	\arrow["b", from=2-2, to=1-1]
	\arrow["h"', from=1-3, to=2-2]
	\arrow["g"', from=2-4, to=1-3]
	\arrow["i", from=2-2, to=2-4]
\end{tikzcd}, \quad I=\langle ab, cd, ef, gd, hi, ga-fh, ai-de, bc-ig\rangle\]
Then $C$ is a dominant Auslander-regular algebra of global dimension 4 and dominant dimension 2.
\end{example}

Recall that the bounded derived category $D^b(\mod A)$ of a finite dimensional algebra $A$ of finite global dimension has a Serre functor $S$ and $A$ is called \emph{fractionally Calabi-Yau} if there exists integers $\ell >0, m$ such that $S^{\ell}$ and the shift $[m]$ are isomorphic as functors on $D^b(\mod A)$. In this case, $A$ is called $\frac{m}{\ell}$-Calabi-Yau, see for example \cite[18.6]{Ye}. If the defining isomorphism of functors is just taken up to a twist by an algebra automorphism, $A$ is called \emph{twisted fractionally Calabi-Yau}, we refer for example to the recent article \cite{CDIM} for more on (twisted) fractionally Calabi-Yau algebras and many examples and applications.
We show next that the existence of a mixed cluster tilting module in the trivial extension algebra $T(A)$ of an algebra $A$ can be seen as property stronger than the twisted fractionally Calabi-Yau property of an algebra $A$.

We will use the following main result of \cite{CDIM}. Recall that an algebra $A$ is said to have complexity at most one if every $A$-module has complexity at most one.
\begin{theorem}\cite{CDIM} \label{CDIMresult}
Let $A$ be a finite dimensional algebra.
Then the following are equivalent:
\begin{enumerate}[\rm(1)]
\item The trivial extension $T(A)$ of $A$ has complexity at most one.
\item $A$ is twisted fractionally Calabi-Yau.
\end{enumerate}
\end{theorem}
We now obtain the following corollary:
\begin{corollary}
Let $A$ be an algebra such that the trivial extension $T(A)$ has a mixed cluster tilting module. Then $A$ is twisted fractionally Calabi-Yau.
\end{corollary}

\begin{proof}
By Theorem \ref{ErdmannHolmgeneralisation}, having a mixed cluster tilting module implies that the self-injective algebra $T(A)$ has complexity at most one. The claim now follows using theorem \ref{CDIMresult}.
\end{proof}
As a consequence of the previous corollary and Theorem \ref{B+N}, $d$-representation-finite algebras are twisted fractionally Calabi-Yau -- as was first shown in \cite{HI1}.

We pose the following question as a possible generalisation of our main result Theorem \ref{B+N}.

\begin{question}
Let $A$ be twisted fractionally Calabi-Yau of finite global dimension. Does the trivial extension of $A$ has a mixed cluster tilting module?
\end{question}
In forthcoming work, some evidence for a positive answer to this question is obtained.

We discuss now when the mixed cluster tilting $B$-module in Theorem \ref{B+N} is an $d'$-cluster tilting $B$-module for some $d'$.

Let $d$ be a positive integer, and $A$ a finite dimensional algebra of global dimension at most $d$. Then the following numbers are equal \cite{I7}, which we call the \emph{$\tau_d$-nilpotency}.
\begin{itemize}
\item The minimal integer $\ell\ge1$ satisfying $\tau_d^\ell(DA)=0$,
\item The minimal integer $\ell\ge1$ satisfying $\tau_d^\ell=0$ as an endofunctor of $\mod A$,
\item The minimal integer $\ell\ge1$ satisfying $\tau_d^{-\ell}(A)=0$,
\item The minimal integer $\ell\ge1$ satisfying $\tau_d^{-\ell}=0$ as an endofunctor of $\mod A$,
\end{itemize}
The $\tau_d$-nilpotency of $A$ is $1$ if and only if $A$ has global dimension at most $d-1$.
If $A$ is $d$-representation-finite with $d$-cluster tilting module $M$, then the following number also equals to the $\tau_d$-nilpotency.
\begin{itemize}
\item The maximal cardinality of $\tau_d$-orbits in $\ind(\add M)$.
\end{itemize}
In particular, for an algebra $A$ with cluster tilting module $M$, having $\tau_d$-nilpotency at most $2$ means that every indecomposable summand of $M$ is projective or injective, i.e. $\add M=\add A\oplus DA$.

The next corollary gives a characterisation when the mixed cluster tilting module $T(A)\oplus N$ given in Theorem \ref{B+N} is an $n$-cluster tilting module for some $n\ge1$. For a finite dimensional algebra $A$ and $P\in\ind(\proj A)$, let
\[t^P=\min \{ l \geq 1 \mid\mbox{$\nu^l(P)$ is not projective}\}\]
We say that $A$ has \emph{constant $\nu$-index} if $t^P$ is constant for all $P\in\ind(\proj A)\setminus\ind(\inj A)$. 

\begin{corollary}\label{strongnrepcorollary}
With the setting as in Theorem \ref{B+N}, the following conditions are equivalent.
\begin{enumerate}[\rm(1)]
\item $B \oplus N$ is an $n$-cluster tilting $B$-module for some positive integer $n$.
\item $A$ has $\tau_d$-nilpotency at most $2$ and constant $\nu$-index. 
\end{enumerate}
In such a case, we have $n=d+t$ where $t$ is the constant $\nu$-index.
\end{corollary}

\begin{proof}
By Theorem \ref{B+N}, the condition (1) is equivalent to the following condition.
\begin{enumerate}[\rm(3)]
\item $t_X$ is constant for all $X\in\ind(\add N)$.
\end{enumerate}
We prove (2)$\Rightarrow$(3). For each $X\in\ind(\add N)$, $\tau_d^A(X)$ is projective since $A$ has $\tau_d$-nilpotency at most $2$. Thus $t_X=t^{\tau_d^A(X)}+1$ is constant since $A$ has constant $\nu$-index.

We prove (3)$\Rightarrow$(2). For $X\in\ind(\add N)$, $t_X=0$ if and only if $X\in\add\tau_d^-(A)$. There is at least one such $X$ since $A\notin\inj A$. Therefore the constant in (3) is positive, and hence $\add N=\add \tau_d^-(A)$ holds. Thus $A$ has $\tau_d$-nilpotency at most $2$. Moreover $t_X=t^{\tau_d^A(X)}+1$ implies that $A$ has constant $\nu$-index since every $P\in\ind(\proj A)\setminus\ind(\inj A)$ has a form $P=\tau_d^A(X)$ for some $X\in\ind(\add N)$.
\end{proof}

This motives the following problem:

\begin{problem}
Classify the $d$-representation-finite algebras of $\tau_d$-nilpotency $2$. When do they have constant $\nu$-index?
\end{problem}

For the case $d=1$, it is easy to see that ring-indecomposable representation-finite hereditary algebras of $\tau$-nilpotency $2$ are precisely the path algebras of type $A_2$ and $A_3$ with non-linear orientations.
We remark that $d$-representation-finite algebras with $\tau_d$-nilpotency $2$ were studied in \cite{HZ1,HZ2,HZ3}.
We will see many new examples for higher $d$ in forthcoming work using the theory of Koszul duality \cite{CIM2}.

We present some examples of $d$-representation-finite algebras with $\tau_d$-nilpotency $2$ in the following.

\begin{example}
Let $A$ be the Nakayama algebra with quiver
$Q$: 
\[\begin{tikzcd}
	1 & 2 & 3 & 4
	\arrow["a", from=1-1, to=1-2]
	\arrow["b", from=1-2, to=1-3]
	\arrow["c", from=1-3, to=1-4]
\end{tikzcd}\]
and relations given by $I= \langle abc \rangle$ with trivial extension algebra $B$.
Then $A$ is 2-representation-finite with a $2$-cluster tilting module $A \oplus I_1 \oplus I_2$ and $\tau_2$-nilpotency $2$.
Moreover $A$ has constant $\nu$-index and thus $B=T(A)$ has a 4-cluster tilting module $B \oplus I_1 \oplus I_3$.
\end{example}

\begin{example}
Let $A=KQ/I$ be a finite dimensional algebra given by 
\[\begin{tikzcd}
	8 &&6&  \\
	& 7 & 2 & 4 & 1 \\
	3 &&5& 
	\arrow["{a_9}", from=1-1, to=2-2]
	\arrow["{a_2}"', from=3-1, to=2-2]
	\arrow["{a_6}"', from=2-2, to=2-3]
	\arrow["{a_1}"', from=2-3, to=2-4]
	\arrow["{a_3}"', from=2-4, to=2-5]
	\arrow["{a_8}", from=2-2, to=1-3]
	\arrow["{a_5}", from=1-3, to=2-4]
	\arrow["{a_7}"', from=2-2, to=3-3]
	\arrow["{a_4}"', from=3-3, to=2-4]
\end{tikzcd}\ \ I=\langle a_1 a_3, a_5 a_3, a_6 a_1-a_7 a_4, a_6 a_1-a_8 a_5, a_2 a_6, a_2 a_7, a_9 a_7, a_9 a_8 \rangle\]
Then $A$ is 3-representation-finite with $\tau_3$-nilpotency $2$ and constant $\nu$-index $t=2$.  Moreover $A$ is in fact the Koszul dual of the stable Auslander algebra of a path algebra of Dynkin type $D_4$. We will study this class of algebras for general Dynkin types in forthcoming work \cite{CIM2}.
\end{example}

The next example shows that $d$-representation-finite algebra with $\tau_d$-nilpotency $2$ does not necessarily have constant $\nu$-index in general.
\begin{example}
Let $A$ be the Nakayama algebra with quiver
$Q$:
\[\begin{tikzcd}
	1 & 2 & 3 & 4 & 5
	\arrow["a", from=1-1, to=1-2]
	\arrow["b", from=1-2, to=1-3]
	\arrow["c", from=1-3, to=1-4]
	\arrow["d", from=1-4, to=1-5]
\end{tikzcd}\]
and relations given by $I=\langle abcd \rangle$ with trivial extension algebra $B$.
Moreover $A$ is $2$-representation-finite with 2-cluster tilting module $A \oplus I_1 \oplus I_2 \oplus I_3$.
Now $\tau_2(I_1)=P_3$ and $\nu_A(P_3)=I_3$ is non-projective, whereas $\tau_2(I_2)=P_4$ and $\nu_A(P_4)=I_4=P_1$ is projective.
Thus the $B$-module $B \oplus I_1 \oplus I_2 \oplus I_3$ is not cluster tilting, but only mixed cluster tilting.
\end{example}

\section{Mixed precluster tilting for gendo-symmetric algebras}\label{section:example mixed pct}

The next theorem gives an iterative construction of mixed precluster tilting modules.
For clarity, given a mixed precluster tilting $A$-module $M$ and $X\in \ind(\add M)$, we use $\ell_X^M$ and $r_X^M$ to denote the associated numbers that were notated as $\ell_X$ and $r_X$ respectively.\footnote{Recall from Proposition \ref{B to A 0} that we have $\idim _BP_X = r_X^M+1$ and $\pdim_BI_X=\ell_X^M+1$, where $B:=\End_A(M)$ and for $P_X$ and $I_X$ are the projective and injective $B$-modules corresponding to $X$ respectively. We specifically clarify this as the condition (b) and (c) looks similar but is different from this formula.}

\begin{definition}\label{define splendid}
Let $A$ be an algebra. We call $A$ \emph{splendid} if the following conditions are satisfied.
\begin{enumerate}[\rm(a)]
\item $M:=A\oplus DA$ is a mixed precluster tiliting $A$-module.
\item For each $P\in\ind(\proj A)\setminus\inj A$, we have $\idim_AP=r_P^M+1$.
\item For each $I\in\ind(\inj A)\setminus\proj A$, we have $\pdim_AI=\ell_I^M+1$.
\end{enumerate}
\end{definition}

We give a family of examples of splendid algebras.
Recall from \cite{FanKoe} that an algebra $A$ is \emph{gendo-symmetric} if $A\simeq \End_B(M)$ for some generator(-cogenerator) $M$ over a symmetric algebra $B$. Recall that the invariants $\ell_X^M$ and $r_X^M$ are always strictly positive.

\begin{proposition}\label{cor:gendosym SGC}
Each gendo-symmetric algebra $A$ that is dominant Auslander-Gorenstein
is splendid.
\end{proposition}

\begin{proof}
Recall from \cite[Prop 2.11]{FanKoe} and \cite[Prop 4.3]{Mar} that, for any $X\in \mod A$ with dominant dimension at least two, one has
\begin{align}
\domdim X =\inf\{i\geq 1 \mid \Ext_A^i(DA,X)\neq 0\}+1 \;\; \text{ and } \;\; \tau^-(X) \simeq \Omega^{-2}(X). \label{eq:gendosymm facts}
\end{align}
Dual formulas hold for modules $X$ of codominant dimension at least two.

Consider $P\in\ind(\proj A)\setminus\ind(\inj A)$ and let
\[ 0 \to P\to I^0 \to \cdots\to I^d\to I^{d+1} \to 0
\]
be its minimal injective coresolution.
Applying \eqref{eq:gendosymm facts} with $X=P$ yields $\Ext_A^i(DA,P)=0$ for all $1\le i< d$, and $\tau_d^-(P) =  \Omega^{-(d+1)}(P) = I^{d+1}\in\inj A$ respectively.
This shows that $M=DA\oplus A$ is a mixed precluster tilting with $r_P^{M}=d$.
Hence, condition (a) and (b) of Definition \ref{define splendid} is satisfied.  Condition (c) follows from a dual argument.
\end{proof}

The following result gives an inductive construction of splendid algebras.

\begin{theorem}\label{thm: iterative construction}
Let $A$ be a dominant Auslander-Gorenstein algebra which is splendid.
Then $B:=\End_A(A\oplus DA)$ is again a dominant Auslander-Gorenstein algebra which is  splendid such that
\[
 r_{\Hom_A(A\oplus DA,P)}^{B\oplus DB} = r_P^{A\oplus DA}  \text{ and } \ell_{D\Hom_A(I,A\oplus DA)}^{B\oplus DB} = \ell_I^{A\oplus DA}
\]
for all $P\in\ind(\proj A)\setminus\inj A$  and $I\in\ind(\inj A)\setminus\proj A$.
In particular, if $A$ is, in addition, minimal Auslander-Gorenstein, then so is $B$.
\end{theorem}
\begin{proof}
Up to Morita equivalence, we can replace the mixed precluster tilting module $A\oplus DA$ by the basic additive generator $M$ of $\add A\oplus DA$.

(i) Since $M$ is mixed precluster tilting, it follows from Theorem \ref{main correspondence 0} that $B=\End_A(M)$ is dominant Auslander-Gorenstein.

(ii) We explain some facts needed to show that $B\oplus DB$ is mixed precluster tilting now. 

Let us write $P_X:=\Hom_A(M,X)\in\ind(\proj B)$ for the indecomposable projective $B$-module corresponding to $X\in \ind(\add A\oplus DA)$.
Recall from Proposition \ref{prop:DCP}(3) that $P_X$ is injective if so is $X$, so we only need to focus on checking Definition \ref{mixed 0} is satisfied in the case when $X\in \ind(\proj A)\setminus \ind(\inj A)$.

By condition (b) in Definition \ref{define splendid}, we can take a minimal injective resolution
\begin{align}\label{inj res of P}
0\to X\to I^0 \to I^1 \to \cdots \to I^r\to I^{r+1}\to0,
\end{align}
such that $r:= r_X^M$ and $I^i\in\proj A$ for all $0\le i\le r$. 

Since $\Ext^i_A(A\oplus DA,X)=0$ for each $1\le i<r$ and \eqref{eq:tau=derived nu}, applying $\nu^{-1}=\Hom_A(DA,-)$ to \eqref{inj res of P} yields an exact sequence
\begin{align}\label{nu^- of inj res of P}
0\to\nu^{-1}(X)\to\nu^{-1}(I^0)\to\nu^{-1}(I^1)\to\cdots\to\nu^{-1}(I^r)\to \tau_r^{-A}(X) \to 0.
\end{align}
with $\nu^{-1}(I^i)\in\proj A$ for each $0\le i\le r$ and $\tau_r^{-A}(X)\in\inj A$. Thus $\nu^{-1}(I^i)\in\proj A\cap\inj A$ for each $0\le i\le r$ since $A$ is dominant Auslander-Gorenstein. By condition (c) in Definition \ref{define splendid}, we have
\[\pdim_A\tau_r^{-A}(X)=\ell^M_{\tau_r^{-A}(X)}+1=r+1.\]
By \eqref{nu^- of inj res of P}, we have
\begin{equation}\label{bijective}
Y:=\nu^{-1}(X)\in\proj A
\end{equation}
By condition (b) in Definition \ref{define splendid} and \eqref{nu^- of inj res of P}, we have
\begin{equation}\label{r+1}
r_Y^M+1 = \idim_AY=r+1.
\end{equation}
On the other hand, as in the proof of Proposition \ref{B to A 0}, applying $\Hom_A(M,-)$ to \eqref{inj res of P} yields a minimal injective coresolution of $P_X$
\begin{align}\label{(M,-) of inj res of P}
0\to\ P_X \to {}_A(M,I^0)\to {}_A(M,I^1)\to\cdots\to {}_A(M,I^r) \to D\Hom_A(\tau_r^{-A}(X),M)\to 0.
\end{align}
with $D\Hom_A(\tau_r^{-A}(X),M)=I_{\tau_r^{-A}(X)}\in\ind(\inj B)\setminus\ind(\proj B)$, which also gives a minimal projective resolution of $I_{\tau_r^{-A}(X)}$.


(iii) By applying $\Hom_A(M,-)$ to the minimal injective resolution \eqref{inj res of P} of $X\in\proj A$, we got an exact sequence \eqref{(M,-) of inj res of P} with $D\Hom_A(\tau_r^{-A}(X),M)=I_{\tau_r^{-A}(X)}\in\ind(\inj B)\setminus\ind(\proj B)$. By the same argument, by applying $\Hom_A(M,-)$ to the minimal injective resolution \eqref{nu^- of inj res of P} of $Y\in\proj A$ with $r_Y^M=r$, see \eqref{bijective} and \eqref{r+1}, we get an exact sequence
\begin{align}\label{(M, ) of nu^- of inj res of P}
0\to P_Y\to \Hom_A(M,\nu^{-1}(I^0))\to\cdots\to\Hom_A(M,\nu^{-1}(I^r))\to D\Hom_A(\tau_r^{-A}(Y),M)\to 0
\end{align}
with $D\Hom_A(\tau_r^{-A}(Y),M)=I_{\tau_r^{-A}(Y)}\in\ind(\inj B)\setminus\ind(\proj B)$.

(iv) Now we show $\Ext_B^i(DB,P_X)=0$ for $1\le i< r$ and $\tau_r^{-B}(P_X)=I_{\tau_r^{-A}(Y)}\in \inj B$.

Notice that $\Ext_B^i(DB,P_X)$ is the homologies of the complex
\begin{align}\label{(DB,-) of (M,-) of inj res of P}
\Hom_B(DB,{}_A(M,I^0))\to\Hom_B(DB,{}_A(M,I^1))\to\cdots\to\Hom_B(DB,{}_A(M,I^r))
\end{align}
obtained by applying $\Hom_B(DB,-)$ to \eqref{(M,-) of inj res of P}. 
Since we have
\[DB\otimes_BM=D\Hom_{B^\op}(M,B)\simeq D\Hom_{B^\op}(\Hom_A(A,M),\End_A(M))\simeq D\Hom_A(M,A),\]
we have
\begin{align*}
&\Hom_B(DB,{}_A(M,I^i))\simeq\Hom_A(DB\otimes_BM,I^i)\simeq\Hom_A(D\Hom_A(M,A),I^i)\\
\simeq&\Hom_{A^\op}(DI^i,\Hom_A(M,A))\simeq\Hom_{A^\op}(DI^i,A)\otimes_A\Hom_A(M,A)\simeq\Hom_A(M,\nu^{-1}(I^i)).
\end{align*}
Thus the complex \eqref{(DB,-) of (M,-) of inj res of P} is isomorphic to the exact sequence \eqref{(M, ) of nu^- of inj res of P}
obtained by applying $\Hom_A(M,-)$ to \eqref{nu^- of inj res of P}.
Thus $\Ext^i_B(DB,P_X)=0$ for each $1\le i<r$.
Finally, by definition (see \eqref{eq:tau=derived nu}), $\tau_r^{-B}(P_X)$ is the cokernel of the right-most morphism of \eqref{(DB,-) of (M,-) of inj res of P}, and hence the right-most term $D\Hom_A(\tau_r^{-A}(Y),M)=I_{\tau_r^{-A}(Y)}\in\ind(\inj B)\setminus\ind(\proj B)$ of \eqref{(M, ) of nu^- of inj res of P}, as desired.

(v) By (iv), $B\oplus DB$ is a mixed precluster tilting $B$-module with $r_{P_X}^{B\oplus DB} = r = r_X^M$, hence condition (b) in Definition \ref{define splendid} is satisfied.  The statement for the $\ell$-side follows from the duality between $r$-numbers and $\ell$-numbers (Definition \ref{mixed 0}).

Finally, the claim about minimal Auslander-Gorenstein is straightforward as we have $r_P^M$ constant over all $P\in\ind(\proj A)\setminus\ind(\inj A)$ and likewise for $\ell_I^M$.
\end{proof}

Combining Proposition \ref{cor:gendosym SGC} and Theorem \ref{thm: iterative construction}, we obtain the following result claiming that all iterated SGC-extensions of gendo-symmetric dominant Auslander-Gorenstein algebras are splendid.

\begin{corollary} \label{gendosymtheorem}
Let $A=A_0$ be a gendo-symmetric algebra. For each $i\ge0$, let $M_i:=A_i \oplus DA_i$ and $A_{i+1}:=\End_{A_i}(M_i)$.
\begin{enumerate}[\rm(1)]
\item Assume that $A$ is a dominant Auslander-Gorenstein algebra. Then for each $i\ge0$, $A_i$ is a dominant Auslander-Gorenstein algebra with a mixed precluster tilting module $M_i$.
\item Assume that $A$ is a minimal Auslander-Gorenstein algebra. Then for each $i\ge0$, $A_i$ is a minimal Auslander-Gorenstein algebra with a precluster tilting module $M_i$.
\end{enumerate}
\end{corollary}


In the rest of this section, we give examples of dominant Auslander-Gorenstein algebras which are splendid.
For the next proposition call a subset $S$ of $\mathbb{Z}/n \mathbb{Z}$ \emph{isolated} if it has the property that with $s \in S$ we have $s \pm 1 \notin S$. 

\begin{proposition}
Let $A$ be a symmetric algebra and $Y$ an indecomposable non-projective $A$-module with $\Omega^n(Y) \simeq Y$ for some $n \geq 1$ and $\Ext_A^i(Y,Y)=0$ for each $1\le i\le n-2$.
Let $S$ be an isolated subset of  $\mathbb{Z}/n \mathbb{Z}$ with representatives in $\{0,...,n-1\}$.
\begin{enumerate}[\rm(1)]
\item $M:=A \oplus \bigoplus_{s \in S}^{}{\Omega^s(Y)}$ is a mixed precluster tilting $A$-module.
\item $\End_A(M)$ is a dominant Auslander-Gorenstein algebra which is splendid.
\end{enumerate}
\end{proposition}
\begin{proof}
When $S = \emptyset$, then the assertions are clear.
Thus assume $S \neq \emptyset$ in the following.

(1) We recall that in a symmetric algebra we have $\tau \simeq \Omega^2$ and hence $\underline{\mod}A$ is a $(-1)$-Calabi-Yau triangulated category.

Let $X= \Omega^r(Y)$ with $r \in S$ and $\ell_X=\min \{ i \geq 1 \mid i+r+1\in S\}$. Using $\tau=\Omega^2$, we obtain 
\[\tau_{\ell_X}(X)=\Omega^{\ell_X+1+r}(Y)\in\add M.\]
It remains to prove that $\Ext^i_A(X,M)=0$ for all $1\le i< \ell_X$.
By dimension shifting and $(-1)$-Calabi-Yau property, we have
\begin{align*}
\Ext_A^i(Y,Y) &\simeq \underline{\Hom}_A(\Omega^i(Y),Y) \simeq\left\{\begin{array}{ll} \underline{\Hom}_A(Y,Y) \neq 0&i\equiv 0\ \mod n,\\
\underline{\Hom}_A(\Omega^{-1}(Y),Y)\simeq
D\underline{\Hom}_A(Y,Y)\neq 0&i\equiv -1\ \mod n.\end{array}\right.
\end{align*}
Thus $\Ext_A^i(Y,Y) \neq 0$ if and only if $i \equiv -1$ or $i \equiv 0$ mod $n$. Now we have
$$\Ext_A^i(X,M)=\Ext_A^i(\Omega^r(Y), \bigoplus_{s \in S}^{}{\Omega^s(Y)})=\bigoplus_{s \in S}^{}{\Ext_A^i(\Omega^r(Y),\Omega^s(Y))}=\bigoplus_{s \in S}^{}{\Ext_A^{i+r-s}(Y,Y)}.$$
Since $S$ is isolated, $i+r-s\not\equiv-1,0$ mod $n$ for each $1\le i< \ell_X$. Hence, this equals to $0$.

(2) This follows from (1) and Proposition \ref{cor:gendosym SGC}.
\end{proof} 

\begin{example}
Let $A$ be a connected symmetric Nakayama algebra with $n$ simple modules.
Then any simple $A$-module $S$ satisfies $\Omega^{2n}(S) \simeq S$ and $\Ext_A^i(S,S)=0$ for $1\le i\le 2n-2$.
Thus the previous proposition can be applied in this case and leads to a large class of dominant Auslander-Gorenstein algebras.
More generally, a smiliar process works for any Brauer tree algebra and leads to the recently introduced gendo Brauer tree algebras in \cite{CM}. This class of algebras were one of our motivating examples for the study of mixed precluster tilting modules.
\end{example}

We give an explicit example of a gendo-symmetric algebra that is Auslander-Gorenstein and that appears in various contexts of abstract algebra such as representation-finite blocks of Schur algebras.

\begin{example}
In \cite[Section 6.1]{CM}, it was shown that the following special gendo-Brauer tree algebra $A=KQ/I$ is higher Auslander, where $Q$ is given by 
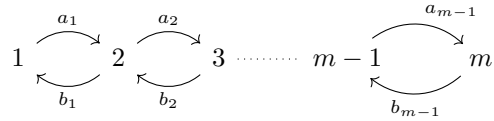
\[\begin{tikzcd}
	1 \ar[r,bend left=40,"a_1"]& 2\ar[r,bend left=40,"a_2"]\ar[l,bend left=40,"b_1"] & 3\ar[l,bend left=40,"b_2"] \ar[r,dotted,no head] & m-1\ar[r,bend left=40,"a_{m-1}"] & m\ar[l,bend left=40,"b_{m-1}"]
\end{tikzcd}\]
and the relations are given by $I=\langle
b_{m-1} a_{m-1}, b_{i-1}a_{i-1}-a_i b_i, a_{i-1}a_i, b_i b_{i-1} \rangle$.
By Corollary \ref{gendosymtheorem}, we obtain that all SGC-extension algebras of $A$ are minimal Auslander-Gorenstein.

We explicitly display how the SGC-extensions of $A$ look like for $m=2$, when $A$ is a Nakayama algebra.
Recall that the Kupisch series of a Nakayama algebra with $n$ simple modules is simply the list $[\LL(P_1),...,\LL(P_n)]$ of Loewy lengths $\LL(P_i)$ of the indecomposable projective $A$-modules and this list uniquely determines quiver and relations of a Nakayama algebra given by quiver and relations.
Thus for $m=2$, $A=A_0$ is the Nakayama algebra with Kupisch series [2,3] and we leave it to the reader to show that the $n$-th SGC-extension $A_n$ of $A$ is the Nakayama algebra with $2+n$ simple modules and Kupisch series $[n+2,n+2,....,n+2,n+3]$ (all entries are equal to $n+2$, except the last one). In terms of quiver and relations, $A_n$ has the quiver
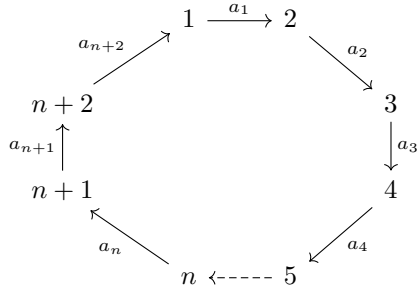
\[\begin{tikzcd}
	& 1 & 2 \\
	{n+2} &&& 3 \\
	{n+1} &&& 4 \\
	& n & 5
	\arrow["{a_1}", from=1-2, to=1-3]
	\arrow["{a_2}", from=1-3, to=2-4]
	\arrow["{a_3}", from=2-4, to=3-4]
	\arrow["{a_4}", from=3-4, to=4-3]
	\arrow[dashed, from=4-3, to=4-2]
	\arrow["{a_n}", from=4-2, to=3-1]
	\arrow["{a_{n+1}}", from=3-1, to=2-1]
	\arrow["{a_{n+2}}", from=2-1, to=1-2]
\end{tikzcd}\]
and the relations of $A_n$ are given by all paths of length $n+2$ starting at a point $i$ for $i$ with $1 \leq i \leq n+1$.
\end{example}


\section{Gluing dominant Auslander-Gorenstein algebras} \label{section: glueing}

Recall that a quiver algebra is called \emph{quadratic} if its relations are quadratic and homogeneous. 

\begin{definition}
Let $n$ be a natural number.
\begin{enumerate}[\rm(i)]
\item An \emph{integer composition} of $n$ is a list of positive integers $a=[a_1,...,a_r]$ with $\sum_{i=1}^{r}{a_i}=n$. Denote by $a_{\leq k}:=\sum_{i=1}^k a_i$, then the \emph{descent set} of $a$ is defined as the subset $D_a=\{a_{\leq 1}=a_1, a_{\leq 2},...,a_{\leq r-1} \}$ of $\{1,2,...,n-1 \}$.
The \emph{complement} $a^\perp$ of $a$ is defined as the unique integer composition with sum $n$ such that $D_{a^\perp}=\{1,2,...,n-1 \} \setminus D_a$.  
\item For a composition $a=[a_1,\ldots, a_r]$ of $n$, denote by $N_a:=K A_{n+1}/I_a$ the quadratic linear Nakayama algebra given by linear $A_{n+1}$ quiver with $I_a$ generated by paths $a_{\leq k} \to a_{\leq k}+1 \to a_{\leq k}+2$ for all $1\leq k < r$:
\[
\begin{tikzcd}
1 \ar[r] & 2\ar[r] & \cdots \ar[r] & a_{\leq k} \ar[r,""{name=p}] & a_{\leq k}+1 \ar[r,""{name=pp}] & a_{\leq k}+2 \ar[r] & \cdots \ar[r]& n+1.
\ar[dashed, no head, bend left, from=p, to=pp]
\end{tikzcd}
\]
\end{enumerate}
\end{definition}

All quadratic linear (connected) Nakayama algebras are of this form; it is clear that there are $2^{n-1}$ of them (enumerated by the choice where the relations appear).
Recall that $\LL(M)$ denotes the Loewy length of a module $M$.
The Kupisch series $[ \LL(P_1), \ldots, \LL(P_{n+1}) ]$ of $N_a$ is given by $[a_1+1,a_1,...,2,a_2+1,a_2,...,2,...,a_r+1,a_r,...,2,1]$.

\begin{proposition}\label{nakayama dominant}
Let $a$ be a composition of $n+1\geq 3$.
\begin{enumerate}[\rm(1)]
\item $N_a$ is dominant Auslander-regular with $\codomdim I_i = \pdim I_i = \LL(P_i^!)-1$ for all non-projective injective $I_i$, where $\LL(P_i^!)$ is the Loewy length of the indecomposable projective $N_{a^\perp}$-module $P_i^!$ corresponding the vertex $i$.  In particular, $\gldim N_a=\max(a^\perp)$ and $\domdim N_a=\min(a^\perp)$.
\item $N_a$ is higher Auslander if and only if $\min(a^\perp)=\max(a^\perp)$.  In particular, the number of higher Auslander quadratic linear Nakayama algebra is equal to the number of divisors of the natural number $n-1$.
\end{enumerate}
\end{proposition}

\begin{remark}\label{rem:naka koszul}
We remark that $N_{a^\perp}$ is the opposite ring of the quadratic dual $(N_a)^!=KA_{n+1}^\op / I_a^{\perp}$ (hence, Koszul dual) of $N_a$, where $I_a^\perp$ is generated by all length 2 paths $x+2\to x+1 \to x$ such that $x\to x+1\to x+2$ is \emph{not} a relation in $N_a$.  The `duality' between homological dimension and Loewy length is not a coincidence for Nakayama algebra, but a more general phenomenon that we will study in forthcoming work \cite{CIM2}.
\end{remark}

\begin{example}\label{eg:naka1}
Consider $a=[2,3,1]$ so that $N_a$ is given by:
\[
\begin{tikzcd}
1 \ar[r] & 2\ar[r,""{name=a2}] & 3 \ar[r,""{name=a3}] & 4 \ar[r] & 5 \ar[r,""{name=a5}] & 6 \ar[r,""{name=a6}] & 7.
\ar[dashed, no head, bend left=40, from=a2, to=a3]
\ar[dashed, no head, bend left=40, from=a5, to=a6]
\end{tikzcd}
\]
This has $D_a=\{2,5\}$ which means that the indecomposable projective-injective modules are $P_i$ for $i\in 1+(D_a\cup\{0\}) = \{1,3,6\}$.  More precisely, we have
\[ P_1=I_3,\quad  P_3=I_6, \quad P_6=I_7, \]
and they have Loewy lengths $3, 4, 2$ respectively.
Now we have $D_{a^\perp}=\{1,2,\ldots,6\}\setminus D_a = \{1,3,4,6\}$, and this yields $a^\perp = [1,2,1,2]$.
Add $1$ to each element of $D_{a^\perp}$ then the resulting set enumerates the indecomposable non-injective projective modules, whose injective coresolutions
are shown in the left-hand side as follows; the right-hand side show the corresponding Gorenstein condition.
\[\begin{tikzcd}[row sep=0]
0 \ar[r] & P_2 \ar[r] & I_3 \ar[r] & I_1 \ar[r] & 0 & & \Ext_A^k(S_1,A)=(S_2^\op)^{\oplus \delta_{k,1}}\\
0 \ar[r] & P_4 \ar[r] & I_6 \ar[r] & I_3 \ar[r] & I_2 \ar[r] & 0 & \Ext_A^k(S_2,A)=(S_4^\op)^{\oplus \delta_{k,2}}\\
0 \ar[r] & P_5 \ar[r] & I_6 \ar[r] & I_4 \ar[r] & 0 & &\Ext_A^k(S_4,A)=(S_5^\op)^{\oplus \delta_{k,1}}\\
0 \ar[r] & P_7 \ar[r] & I_7 \ar[r] & I_6 \ar[r] & I_5 \ar[r]& 0 & \Ext_A^k(S_5,A)=(S_7^\op)^{\oplus \delta_{k,2}}
\end{tikzcd}
\]
Notice how entries of $a^\perp$ corresponds to $\idim P_i$ for $i\in D_{a^\perp}$.
Also, the grade bijection $S\mapsto D\Ext_A^{\mathrm{grade}S}(S,A)$ is given by $(1,2,4,5,7,6,3)$, where $(\ldots, i_j,i_{j+1},\ldots,)$ is the cyclic permutation that maps $i_j$ to $i_{j+1}$, etc.
\end{example}

The proof of Proposition \ref{nakayama dominant} follows from the next result. We give a construction via gluing that illustrates that in general one can expect that there are much more dominant Auslander-Gorenstein algebras than minimal Auslander-Gorenstein algebras.

Consider two algebras $B:=KQ_B/I_B$ and $C=KQ_C/I_C$.
Choose $n$ sinks in $Q_B$ and $n$ sources in $Q_C$, label them $v_1,\ldots, v_n$ so that we obtain a new quiver $Q$ where the $v_i$'s in the two quivers are identified. 
Define a new algebra
\[
A:=KQ/I=KQ/\langle I_B,I_C,bc\mid b \in B, c\in C\rangle.
\]
Note that $I$ is an admissible ideal and so $A$ is finite-dimensional.  For $\Lambda\in\{A,B,C\}$, we denote by $P_x^\Lambda, I_x^\Lambda$ the indecomposable projective, indecomposable injective $\Lambda$-module corresponding to vertex $x$ in the quiver of $\Lambda$, respectively.
\begin{proposition}\label{gluing}
Let $A,B,C$ be algebras as above.
\begin{enumerate}[\rm(1)]
\item If $B$ and $C$ are dominant Auslander-Gorenstein algebras, then so is $A$.
\item In this case, we have
\[
\idim P_x^A = \begin{cases} \idim P_x^B, & \text{if $x\in Q_B$;}\\
\idim P_x^C, & \text{if $x\in Q_C$ and $\pi(x)\notin\{v_i\}_i$;}\\
\idim P_x^C + \idim P_{v_i}^B & \\
 \quad = \pdim I_{v_i}^C + \idim P_{v_i}^B,  & \text{if $x\in Q_C$ and $\pi(x)=v_i$ some $i$,}
\end{cases}
\]
where $\pi$ denotes the canonical permutation associated to $C$.  The same formula holds if we replace $\idim$ by $\domdim$.
\item Let $\mathcal{D}$ be the union of $\{\idim C, \idim P_x^B\mid x\in Q_B\setminus \{v_j\}_j\}$ and $\{\idim P^B_{v_i}+ \pdim I^C_{v_i}\mid 1\leq i\leq n\}$, then we have $\idim A= \max \mathcal{D}$ and $\domdim A=\min\mathcal{D}$.
\end{enumerate}
\end{proposition}

\begin{proof}
For simplicity, we omit the superscript $A$ for $A$-modules.
Let $I_x$ be an indecomposable projective $A$-module corresponding to the vertex $x\in Q$.
Note that if $P_y^\Lambda \in\inj \Lambda$ (resp. $I_y^\Lambda\in \proj\Lambda$) for any $\Lambda\in\{B,C\}$, then $P_y\in\inj A$ (resp. $I_y\in \proj A$).

If $x \in Q_C$, then the minimal projective resolution for $I_x^C$ lifts to that of $I_x$ in $\mod A$ naturally (i.e. we can just remove the superscript $C$ everywhere).

Suppose that $x \in Q_B$ and let $p_x:= \pdim I_x^B$.
Since $B$ is dominant Auslander-Gorenstein, we have $\Omega_B^{p_x}(I_x^B)\simeq P_y^B$ for some $y\in Q_B$.  If $y\notin \{v_i\}_i$, then the full minimal projective resolution of $I_x^B$ in $\mod B$ lifts to that of $I_x$ in $\mod A$.
Otherwise, splicing with the (lift of the) minimal projective resolution of $I_y^C$ with the that of $I_x^B$ yields the full minimal projective resolution of $I_x$.
Since $B$ and $C$ are domiant Auslander-Gorenstein, all terms in these projective resolutions are projective-injective except the last one.
The claimed formulae are now clear from these resolutions.
\end{proof}

Quadratic linear Nakayama algebras are all obtained from iteratively gluing hereditary Nakayama algebras and in this way one obtains a proof of \ref{nakayama dominant} as a special case of \ref{gluing}.

\begin{remark}
Let 
\begin{align}\label{BCA}
\widetilde{B}:=\begin{pmatrix}
B & {}_BM_\Lambda \\ 0 & \Lambda 
\end{pmatrix}, \quad \widetilde{C}:=\begin{pmatrix}
\Lambda & {}_\Lambda N_C \\ 0 & C
\end{pmatrix},\ \ \ A:= \begin{pmatrix}
B & M & 0 \\ 0 & \Lambda & N \\ 0 & 0 & C
\end{pmatrix}.
\end{align}
Proposition \ref{gluing} can be generalised as follows:
Suppose we have two dominant Auslander-Gorenstein algebras $\widetilde{B}$ and $\widetilde{C}$ of the form \eqref{BCA}.
If $\Lambda$ is self-injective and both $M,N$ are sincere as $\Lambda$-modules, then the algebra $A$ given by \eqref{BCA}
is also dominant Auslander-Gorenstein.
\end{remark}

\section*{Acknowledgements} 
This project profited from the use of the GAP-package \cite{QPA}. A part of this work was done when the authors attended the workshop ``Representation Theory of Quivers and Finite Dimensional Algebras'' in Mathematisches Forschungsinstitut Oberwolfach.  AC thanks also  Universit\"{a}t Stuttgart for their hospitality while part of this work was done.


\end{document}